\documentclass[12pt,reqno]{amsart}
\usepackage{amssymb}
\usepackage{cases}
\usepackage{bbm}
\usepackage{mathrsfs}
\usepackage{graphicx}
\usepackage{amsfonts}
\usepackage{enumerate}
\usepackage{hyperref}

 \usepackage[dvipsnames]{xcolor}

\usepackage{amssymb, amsmath, amsfonts, latexsym}
\usepackage{enumerate}
\usepackage{color}

\newtheorem{thm}{Theorem}
\newtheorem{cor}{Corollary}
\newtheorem{lem}{Lemma}
\newtheorem{prop}{Proposition}
\newtheorem{example}{Example}
\newtheorem{remarks}{Remark}

\newtheorem{defn}{Definition}
\newtheorem{hyp}{Hypothesis}
\numberwithin{equation}{section}

\date{}

\def\vep{\varepsilon}
\def\<{\langle}
\def\>{\rangle}

\def\d"{^{\prime\prime}}

\def\bhyp{\begin{hyp}}
\def\nhyp{\end{hyp}}

\def\bdef{\begin{defn}}
\def\ndef{\end{defn}}

\def\bthm{\begin{thm}}
\def\nthm{\end{thm}}

\def\bprop{\begin{prop}}
\def\nprop{\end{prop}}

\def\brmk{\begin{remarks}}
\def\nrmk{\end{remarks}}

\def\bexa{\begin{example}}
\def\nexa{\end{example}}

\def\blem{\begin{lem}}
\def\nlem{\end{lem}}

\def\bcor{\begin{cor}}
\def\ncor{\end{cor}}

\def\bexe{\begin{exe}}
\def\nexe{\end{exe}}

\def\bprf{\begin{proof}}
\def\nprf{\end{proof}}

\def\bdes{\begin{description}}
\def\ndes{\end{description}}

\def\benu{\begin{enumerate}}
\def\nenu{\end{enumerate}}

\def\ess{\text{\rm{ess}}}

\usepackage{geometry}
\begin{document}

 \title[Large deviations and quasi-ergodic distribution]
 {Large deviations of the empirical measures of a strong-Feller Markov process inside a subset and  quasi-ergodic distribution}

\author[A. Guillin]{\textbf{\quad {Arnaud} Guillin$^{\dag}$    }}
\address{{\bf Arnaud Guillin}. Universit\'e Clermont Auvergne, CNRS, LMBP, F-63000 CLERMONT-FERRAND, FRANCE}
 \email{arnaud.guillin@uca.fr}

\author[B. Nectoux]{\textbf{\quad Boris Nectoux$^{\dag}$  }}
\address{{\bf Boris Nectoux}.Universit\'e Clermont Auvergne, CNRS, LMBP, F-63000 CLERMONT-FERRAND, FRANCE}
 \email{boris.nectoux@uca.fr}

\author[L. Wu]{\textbf{\quad Liming Wu$^{\dag}$ \, \, }}
\address{{\bf Liming Wu}. Universit\'e Clermont Auvergne, CNRS, LMBP, F-63000  CLERMONT-FERRAND, FRANCE, and,
Institute for Advanced Study in Mathematics, Harbin Institute of Technology, Harbin 150001, China}
\email{Li-Ming.Wu@uca.fr}

\begin{abstract}
In this work, we  establish, for a strong Feller process, the large deviation principle  for  the occupation measure  conditioned not to exit a given subregion.  The rate function vanishes only at a unique measure, which is  the so-called   quasi-ergodic distribution of the process in this subregion.
In addition, we show that the rate function is the Dirichlet form in the particular case when the process is reversible.
 We apply our results to several stochastic processes such as the solutions of  elliptic stochastic differential equations driven by a rotationally invariant $\alpha$-stable process,   the kinetic Langevin process,  and the overdamped Langevin process driven by a Brownian motion.  \\

 \medskip

\centerline {\it Dedicated to Patrick Cattiaux}
\end{abstract}
\maketitle

\vskip 20pt\noindent {\it AMS 2020 Subject classifications.}


\section{Introduction}

\subsection{Setting and purpose of this work}
\begin{sloppypar}
Let $(X_t)_{t\ge   0}$ be a {\it c\`adl\`ag} Markov process valued in a Polish space $\mathscr S$, defined on the filtered probability space
$(\Omega,\mathcal F, (\mathcal F_t)_{t\ge   0}, (\mathbb P_x)_{x\in E})$, where $\mathbb P_x$ means that $\mathbb P_x(X_0=x)=1$ for each $x\in \mathscr  S$.
Given a nonempty  open subset  $\mathscr   D$ of  $\mathscr S$, consider the first exiting time of the process from   $\mathscr   D$:
\begin{equation}\label{exit-time}
\sigma_{\mathscr  D }:=\inf\{t\ge 0, X_t\in \mathscr  D ^c\}.
\end{equation}
A natural question  in population  processes~\cite{collet2012quasi,meleard2012quasi,champagnat2021lyapunov} and in metastability in molecular dynamics~\cite{lelievre2016partial,di-gesu-lelievre-le-peutrec-nectoux-17,DLLN,lelievre2022eyring},   is to investigate the long time behavior of the law of the process $(X_t)_{t\ge   0}$  conditioned to stay inside  $\mathscr  D$, i.e.   to study  the quantity
$$
\lim_{t\to +\infty} \mathbb P_\nu[X_t\in \cdot\ | t<\sigma_{\mathscr D}],
$$
 where $\nu$ is a given initial distribution on ${\mathscr D}$ and $\mathbb P_\nu(\cdot)=\int_{\mathscr S} \mathbb P_x(\cdot) \nu(dx)$ (under $\mathbb P_\nu$, the distribution of $X_0$ is $\nu$). Intuitively the limit distribution $\mu_{\mathscr D}$ should satisfy
$$
\mu_{\mathscr D}(\cdot) = \mathbb P_{\mu_{\mathscr D}}[X_t\in\cdot\ | t<\sigma_{\mathscr D}],\ \forall t>0.
$$
This is exactly the definition of the {\it quasi-stationary distribution} (q.s.d. in short) of the process in $\mathscr  D$, see e.g.~\cite{collet2012quasi}.
Considering the killed semigroup
\begin{equation}\label{killedsg1}
P_t^{\mathscr D}(x,A):=\mathbb P_x[X_t\in A, t<\sigma_{\mathscr D}], \ \forall A\in\mathcal B(\mathscr S),
\end{equation}
where  $\mathcal B(\mathscr S)$ is the Borel $\sigma$-field of $\mathscr S$,
then $\mu_{\mathscr D}$ is a q.s.d. if and only if
$$
\mu_{\mathscr D} P_t^{\mathscr D} = \lambda(t) \mu_{\mathscr D},\ \lambda(t)=\mathbb P_{\mu_{\mathscr D}}(t<\sigma_{\mathscr D}),
$$
i.e. $\mu_{\mathscr D}$ is a positive left eigen-measure of the killed Dirichlet semigroup $P_t^{\mathscr D}$. 
In our previous work \cite{guillinqsd}, we gave  a quite general framework
for the existence, uniqueness of $\mu_{\mathscr D}$,  as well as for  the exponential convergence of $
\mathbb P_\nu[X_t\in \cdot\ | t<\sigma_{\mathscr D}]
$ to $\mu_{\mathscr D}$ as $t\to +\infty$.
From a statistical point of view, it is also very natural to consider the limit behavior of the conditional distribution $\mathbb P_\nu [L_t\in \cdot\ | t<\sigma_{\mathscr D}]$ of the empirical distributions (or occupation measures)
\begin{equation}\label{empirical-L}
L_t=\frac{1}{t}\int_{0}^{t} \delta_{X_s} dx
\end{equation}
as $t\to+\infty$, where $\delta_x$ is the Dirac measure at the point $x$. Quite curiously, the empirical distribution $L_t$, knowing that $\{t<\sigma_{\mathscr D}\}$, will not converge  to the q.s.d. $\mu_{\mathscr D}$ (unlike in the case where  $\mathscr D=\mathscr S$ if $(X_t, t\ge0)$ is ergodic), but to the so-called \textit{quasi-ergodic distribution} (q.e.d. in short) $\pi_{\mathscr D}$ which is defined by
$$\pi_{\mathscr D}=\varphi \mu_{\mathscr D}$$ where $\varphi$ is the right positive eigenfunction satisfying $P_t^{\mathscr D}\varphi =\lambda(t)\varphi$ and $\mu_{\mathscr D}(\varphi)=\int_{\mathscr S} \varphi d\mu_{\mathscr D}=1$.
 \end{sloppypar}

The purpose of this work is to establish the large deviation principle (L.D.P. in short) of $\mathbb P_\nu[L_t\in \cdot\ | t<\sigma_{\mathscr D}]$ with some rate function $I_{\mathscr D}$ which vanishes only at a unique measure which is the q.e.d. $\pi_{\mathscr D}$. This gives quite precise information about how it approaches to the q.e.d. $\pi_{\mathscr D}$ (for instance exponentially fast in probability, see Corollary \ref{co1}).

We should emphasize that the true history is much more delicate than what said roughly above. Indeed $(\lambda(t), \mu_{\mathscr D}, \varphi)$ is in general not unique: even for the one-dimenional Ornstein-Uhlenbeck process with $\mathscr S=\mathbb R$, $\mathscr D=(0,+\infty)$, the uniqueness of the q.s.d. fails, see \cite{lladser2006domain}. Our work consists to find a rich family of initial distribution $\nu$ (covering at least all Dirac measures $\delta_x, x\in \mathscr D$) so that the intuitive picture above holds. As the rate function $I_{\mathscr D}$ of the L.D.P. of $L_t$ is usually interpreted as some (minus)-entropy functional, the q.e.d. $\pi_{\mathscr D}$ where $I_{\mathscr D}$ vanishes can be understood as the quasi-equilibrium of maximal entropy.

\subsection{Organization of this work}
This work is organized as follows. In Section \ref{sec.11}, we introduce  the conditions we will impose on the process $(X_t,t\ge 0)$ which are mainly those of \cite{guillinqsd} adapted to the killed Feynman-Kac semigroups we consider. We then give in Section \ref{sec.MR} the main result of this work which is    Theorem \ref{thm-main2} about the large deviations of $\mathbb P_\nu[L_t\in \cdot\ | t<\sigma_{\mathscr D}]$ on $\mathcal P(\mathscr D)$ equipped with the $\tau$-topology (stronger than the usual weak convergence topology), see also Theorem \ref{thm-main1} about the spectral gap of the killed Feynman-Kac semigroup. We also provide   the identification of the rate function in the reversible case, see Corollary \ref{cor22}.
Section \ref{sec.Pr} is devoted to the proof of Theorem \ref{thm-main1}, and we prove Theorem \ref{thm-main2} in Section \ref{secPr2}. In Section \ref{sec.Pr3} we prove the identification of the rate function in the reversible case and finally provides examples in Section \ref{pr.Ex}.

\section{Main results}

\subsection{Framework: notations and assumptions}
\label{sec.11}
Let $(X_t,t\ge  0)$ be a time homogeneous Markov process valued in a metric complete separable (say Polish) space $\mathscr S$, with c\`adl\`ag paths and satisfying the strong Markov property, defined on the filtered probability space $(\Omega, \mathcal F, (\mathcal F_t)_{t\ge  0}, (\mathbb P_{x})_{x\in S })$ where $\mathbb P_{x}[X_0= x]=1$ for all $x\in \mathscr S $ (and where the filtration satisfies the usual condition).
Let $\mathcal B(\mathscr S )$ be the Borel $\sigma$-algebra of $\mathscr S $, $b\mathcal B(\mathscr S )$ the Banach space of all bounded and Borel measurable functions $f:\mathscr S \to \mathbb R$ equipped with the sup-norm
 $$\Vert f \Vert=\sup_{ x\in \mathscr S }\vert f( x)\vert.$$
 We also denote by $\mathcal C_b(\mathscr S)$ the space of bounded continuous real-valued functions over $\mathscr S$.
 Given an initial distribution $\nu$ on $\mathscr S $, we write $\mathbb P_\nu(\cdot)=\int_{S } \mathbb P_{x}(\cdot) \nu(d{x}) $.   For $\mathcal A\subset \mathscr S$, we denote by $\mathbf 1_{\mathcal A}$   the indicator function of $\mathcal A$.
The  transition probability semigroup  of $(X_t,t\ge  0)$ is denoted by $(P_t,t\ge  0)$.
 We say that $P_t$ is strong Feller  if  $P_tf$ is continuous on $\mathscr S $ for any $f\in b\mathcal B(\mathscr S )$.
 We denote by $\mathcal P(\mathscr S)$ the space of  probability measures on  $\mathscr S$.
 For any measure $\nu$, transition kernel $P(x,dy)$,  and function $f$ on $(\mathscr S, \mathcal B(\mathscr S))$, we write
 $$\nu(f)=\<\nu, f\>=\int_{\mathscr S} fd\nu,\ Pf(x)=\int_{\mathscr S} f(y)P(x,dy),\  \text{ and } (\nu P)(f)=\nu(Pf).$$
We recall that the space $\mathbb D([0,T],\mathscr S )$ of $\mathscr S $-valued c\`adl\`ag paths   defined on $[0,T]$, equipped with the Skorokhod topology, is a Polish space, see e.g.~\cite{billingsley2013}.

For a continuous time Markov process, often what is given is its generator $\mathfrak L$, not its transition semigroup $(P_t,t\ge  0)$, which is unknown in general. We say that a continuous   function  $f$ belongs to the extended domain $\mathbb D_e(\mathfrak L)$ of $\mathfrak L$, if there is some measurable function $g$ on $\mathscr S $ such that
$\int_0^t |g|(X_s)ds<+\infty$, $\mathbb P_{x}$-a.e. for all ${x}\in \mathscr S $
 and
\begin{equation}\label{eq.de}
 M_t(f):=f(X_t)-f(X_0)- \int_0^t g(X_s)ds
\end{equation}
is a $\mathbb P_{x}$-local martingale for all ${x}\in \mathscr S$. Such a function $g$, denoted by $\mathfrak L f$, is not unique in general. But it is unique up to the equivalence of
{\it quasi-everywhere} (q.e.) that we recall: two functions $g_1,g_2$ are said to be equal q.e., if $g_1=g_2$ almost everywhere in the (resolvent) measure $R_1({x},\cdot)=\int_0^{+\infty} e^{-t}P_t({x},\cdot)dt$ for  every  ${x}\in \mathscr S $.
\medskip

We will work in the following framework, which is slightly different from our previous work \cite{guillinqsd}.  More precisely, we consider the following assumptions on the non-killed process:

\benu
\item[{\bf (C1)}]  For each $t>0$, $P_t$ is strong Feller.

\item[{\bf (C2)}]  For every $T>0$, $x \mapsto\mathbb P_{x}(X_{[0,T]}\in \cdot)$   is continuous from $\mathscr S $ to the space $\mathcal P(\mathbb D([0,T], \mathscr S ))$
equipped with the weak convergence topology.
\nenu


 \benu
\item[{\bf (C3)}]  There exist a continuous  function $\mathbf W : \mathscr S \to [1,+\infty[$, belonging   to the extended domain $\mathbb D_e(\mathfrak L)$,   two sequences of positive constants $(r_n)$ and $(b_n)$ where $r_n\to +\infty$, and an increasing sequence of compact subsets $(K_n)$ of $\mathscr S $ and some constant $p>1$, such that
$$
-\mathfrak L \mathbf W^p ({x}) \ge  r_n \mathbf W^p({x}) - b_n \mathbf 1_{K_n}({x}), \ q.e.
$$
\nenu

Let us now introduce the setting for the killed process. Let $\mathscr D $ be an non empty and open   subset of  $\mathscr S $  and  $\sigma_{\mathscr D}$ be the first exiting time of $\mathscr  D$ defined in  \eqref{exit-time}.
The transition semigroup of the killed process $(X_t, 0\le  t< \sigma_{D })$ is given by \eqref{killedsg1}, or equivalently
for $t\ge  0$ and $ x\in \mathscr D $,
\begin{equation}\label{killedsg}
P_t^{\mathscr D } f( x) = \mathbb E_{x} [\mathbf 1_{t<\sigma_{\mathscr D }} f(X_t)],
\end{equation}
for $f\in b\mathcal B(\mathscr D )$. This semigroup is often called the Dirichlet semigroup.

We now turn   to the conditions on the Dirichlet  semigroup $(P_t^{\mathscr D},t\ge 0)$.

\benu
\item[{\bf (C4)}] For a measure-separable\footnote{Here measure-separability means: if $\mu(f)=\nu(f)$ for all $f\in \mathcal C$, the two positive measures $\mu, \nu$ on $\mathscr D$ are the same.} class $\mathcal C$ of bounded and continuous functions on $\mathscr D$ , $P_t^{\mathscr D}f$ is continuous on $\mathscr D$ for any $f\in \mathcal C$.


\item[{\bf (C5)}]
 There exists $t_0>0$ such that for all $t\ge  t_0$,  for all ${x}\in \mathscr D $ and nonempty open subsets $O$ of $\mathscr D $, $
P^{\mathscr D }_t({x},O)>0$.
Moreover,  there exists $ x_0\in \mathscr D $ such that $\mathbb P_{ x_0}[\sigma_{\mathscr D }<+\infty]>0$.
\nenu

Notice that {\bf (C1),\ (C3)} are slightly stronger than  \cite[{\bf (C1)}, {\bf (C3)}]{guillinqsd}.

\medskip

 \noindent
\textbf{Running Example}.
A prototypical example of reversible dynamics satisfying
  \textbf{(C1)}  $\to$ \textbf{(C5)} is the solution to the so-called overdamped Langevin process
\begin{equation}\label{eq.Lsur}
dy_t= \mathbf c(y_t)dt + dB_t,
\end{equation}
  where $(B_t,t\ge 0)$ is a standard Brownian motion over $\mathbb R^d$.
It can indeed be   checked that
that when  $\mathbf c:\mathbb R^d\to \mathbb R^d$ is  locally Lipschitz   such that
\begin{equation}\label{eq.L1}
\lim_{|x|\to +\infty}\mathbf c(x)\cdot \frac{x}{|x|}= -\infty,
\end{equation}    the conditions  \textbf{(C1)}  $\to$ \textbf{(C5)} are satisfied for the strong solution  $(y_t,t\ge 0)$  of \eqref{eq.Lsur} on any subdomain $\mathscr D$ (i.e. non empty, open and connected) of $\mathbb R^d$ with the
Lyapunov function
$$\mathbf W(x)=e^{a |x|(1-\chi(x))},$$
where  $a>0$ and   $\chi\in \mathcal C^\infty_c(\mathbb R^d)$ equals $1$ in a neighborhood of $0$ in  $\mathbb R^d$. Such a claim can be proved using e.g. the techniques developed in \cite{guillinqsd3}. We refer to Section \ref{pr.Ex} for more complicated examples arising from statistical physics of processes satisfying   \textbf{(C1)}  $\to$ \textbf{(C5)}.

\subsection{Main results}
\label{sec.MR}
In this section, we state the main result of this work, which is Theorem \ref{thm-main2} below. Before, we need  a result  on the spectral gap of  the   killed (outside $\mathscr D$) Feynman-Kac semigroup of $(X_t,t\ge 0)$, which has its own interest. This is the purpose of Theorem \ref{thm-main1} stated in Section \ref{sec.SPg}.  In a nutshell, we need this control on Feynman-Kac semigroup as our approach for large deviations is based on G\"artner-Ellis theorem. It thus relies on a control of a log-Laplace transform which can be recasted in a Feynman-Kac framework.

\subsubsection{Spectral gap of Feynman-Kac semigroup on weighted function spaces}
\label{sec.SPg}

In this section we study  the existence of the spectral gap of the killed Feynman-Kac semigroup of $(X_t,t\ge 0)$ on a weighted space of measures.
In order to introduce our main object of interest, we need to introduce the potential $V$  for which we will suppose throughout the paper without further mention

\smallskip

\benu
\item[{\bf ($\mathbf{H_V}$)}] $V\in b\mathcal B(\mathscr D)$, i.e. $V$ is a bounded and measurable function.
\nenu

\smallskip

The killed (outside $\mathscr D$) Feynman-Kac semigroup is given by
\begin{equation}\label{FKsg}
P_t^{\mathscr D,V} f(x)= \mathbb E_x \Big[ f(X_t)\, e^{\int_{0}^{t} V(X_s) dx}\mathbf 1_{t<\sigma_{\mathscr D}} \Big], f\in b\mathcal B(\mathscr D), t\ge 0, x\in S.
\end{equation}
Note that the  generator of this killed semigroup  is (formally) the Schr\"odinger operator $\mathfrak L_{\mathscr D, V}=\mathfrak L_{\mathscr D} +V$ where $\mathfrak L_{\mathscr D}$ is the generator of $(P_t^{\mathscr D},t\ge 0)$.  Note also that $P_t^{\mathscr D}=P_t^{\mathscr D,0}$.
We refer to the classical textbook~\cite{del2004feynman}  for the theory of  Feynman-Kac semigroups (see also~\cite{nagasawa2012stochastic,del2003particle}).

Under the condition {\bf (C3)}, we will consider $P_t^{\mathscr D,V}$ as bounded operators on the weighted Banach space $b_{\mathbf W}\mathcal B(\mathscr D)$, where $b_{\mathbf W}\mathcal B(\mathscr D)$  is defined by:
\begin{equation}\label{weighted-W}
b_{\mathbf W}\mathcal B(\mathscr D):=\Big\{f:\mathscr D\to \mathbb R \text{ measurable s.t. } \ \|f\|_{\mathbf W}:=\sup_{x\in \mathscr D}\frac{|f(x)|}{{\mathbf W}(x)}<+\infty\Big\}.
\end{equation}
Indeed,  by \cite[Proposition 5.1]{guillinqsd} and {\bf (C3)}, we have using also the fact that ${\mathbf W}\ge   1$:
\begin{align*}
\mathfrak L {\mathbf W}=\mathfrak L ({\mathbf W}^p)^{1/p}\leqslant \frac{1}{p} ({\mathbf W}^p)^{\frac{1}{p} -1} \mathfrak L {\mathbf W}^p&\leqslant \frac{1}{p} {\mathbf W}^{1-p} (-r_n {\mathbf W}^p -b_n \mathbf 1_{K_n}) \\
&\leqslant -r^*_n {\mathbf W} +  b^*_n \mathbf 1_{K_n}
\end{align*}
where $r_n^*= {r_n}/{p}$ and $b_n^*=b_n/p$. Consequently,  $e^{-b_n^* t} {\mathbf W}(X_t)$ is a supermartingale.  Hence, one has that:
$$P_t^{\mathscr D} {\mathbf W} \leqslant e^{b_n^* t}{\mathbf W},$$ and then that:
$$
 P_t^{\mathscr D,V} {\mathbf W}(x) \leqslant  e^{\|V\| t} P_t^{\mathscr D} {\mathbf W}(x)\leqslant   e^{(b_n^*+\|V\| ) t} {\mathbf W}(x).
$$
In conclusion, we have that
$$\|P_t^{\mathscr D,V}\|_{\mathbf W}:=\sup\{\|P_t^{\mathscr D,V} f\|_{\mathbf W}, \ \|f\|_{\mathbf W}\leqslant 1\} \leqslant e^{(b_n^*+\|V\|  )t}.$$
Large deviations of $\mathbb P_\nu[L_t\in\cdot|t<\sigma_{\mathscr D}]$ are closely related to the spectral properties of  the killed Feynman-Kac semigroup $(P_t^{\mathscr D,V},t\ge 0)$, see indeed~\cite{DS89} and~\cite{Wu97}. For this reason, we will need  following result   about the spectral gap of the Feynman-Kac semigroup $P_t^{\mathscr D,V}$ on $b_{\mathbf W}\mathcal B(\mathscr D)$, which generalizes \cite[Theorem 2.2]{guillinqsd}  from the case $V\equiv 0$ to general real-valued bounded $V$.

\bthm\label{thm-main1} Assume  {\rm {\bf (C1)}}$\to$ {\rm {\bf (C5)}} and {\rm {\bf ($\mathbf{H_V}$)}}.
For any given bounded potential $V\in b\mathcal B(\mathscr D)$, consider the log spectral radius of $P^{\mathscr D,V}_1$ on $b_{\mathbf W}\mathcal B(\mathscr D)$:
\begin{equation}\label{sp_radius}
\Lambda_{\mathscr D}(V) =
\lim_{t\to+\infty} \frac{1}{t} \log \|P_t^{\mathscr D,V}\|_{\mathbf W}.
\end{equation}
Then:
\benu
\item[{\rm \textbf 1}.] For any $t>0$,
there is only one probability measure  $\mu_{\mathscr D,V}$ such that $
\mu  P_t^{\mathscr D,V}= c(t) \mu_{\mathscr D,V} $
for some constant $c(t)$ and $\mu_{\mathscr D,V}(\mathbf W)<+\infty$. Moreover, $\forall t>0$, $
c(t)=e^{\Lambda_{\mathscr D}(V) t}$, {$\Lambda_{\mathscr D}(V)<\sup_{\mathscr D} V$},
and $\mu_{\mathscr D,V}$ is independent of $t>0$ and charges all non-empty open subsets $O$ of $\mathscr  D$.

\item[{\rm \textbf 2}.]  There is a unique continuous function $\varphi_{\mathscr D,V}$ bounded by $c\mathbf W$  (for some constant $c>0$) such that  $\mu_{\mathscr D,V}(\varphi_{\mathscr D,V})=1$ and
    \begin{equation}\label{thm-main-aa}
      \ P_t^{\mathscr D,V} \varphi_{\mathscr D,V}= e^{\Lambda_{\mathscr D}(V)t} \varphi_{\mathscr D,V} \text{ on }\mathscr D ,\ \forall t\ge 0.
    \end{equation}
 Moreover, $\varphi_{\mathscr D,V}>0$  everywhere  on $\mathscr D $.

\item[{\rm \textbf 3}.]  There exist $\delta>0$ and $C\ge  1$ such that for all $f\in b_{\mathbf W}\mathcal B(D )$ and  $t>0$:
\begin{equation}\label{thm-maina}
\big |e^{-\Lambda_{\mathscr D}(V)t} P_t^{\mathscr D,V} f - \mu_{\mathscr D,V}(f)\cdot \varphi_{\mathscr D, V} \big | \le C e^{-\delta t} \|f\|_{\mathbf W}\cdot {\mathbf W} \text{ on } \mathscr D.
\end{equation}
\nenu
\nthm

 Note also that Item  \textbf 3  in Theorem \ref{thm-main1} implies  that the  Feynman-Kac operator $P_t^{\mathscr D,V}$ ($t>0$) on $b_{\mathbf W}\mathcal B(D )$ has a spectral gap near its spectral radius $e^{\Lambda_{\mathscr D}(V)t}$ and its spectral projection is the mapping  $f\mapsto \varphi_{\mathscr D,V}\mu_{\mathscr D,V}(f)$, which is one-dimensional. Such a result is of independent interest.  Note that $\mu_{\mathscr D,0}$ is the (unique) q.s.d.   of the Markov process $(X_t,t\ge 0)$  in    
 $$\mathcal P_{\mathbf W}(\mathscr D):=\{\nu \in \mathcal P(\mathscr D), \nu(\mathbf W)<+\infty\}.$$
\medskip

\noindent
\textbf{Remark}. The notion of q.s.d. can also be extended to    killed   renormalized Feynman–Kac semigroup, see e.g.  \cite[Definition 1]{guillinFK}.  Item~\textbf 1  in Theorem \ref{thm-main1}  implies that $ \mu_{\mathscr D,V} $ is the  unique q.s.d.  of the killed (in $\mathscr  D$) renormalized Feynman–Kac semigroup   associated with \eqref{FKsg}   in $\mathcal P_{\mathbf W}(\mathscr D)$. 
\medskip

We will prove Theorem \ref{thm-main1}  in Section \ref{sec.Pr} using  \cite[Theorem 3.5 and Theorem 4.1]{guillinqsd}. To prove the spectral gap of  $P_t^{\mathscr D,V}$ ($t>0$) on $b_{\mathbf W}\mathcal B(D )$ we will use   the   non-compact parameter $\beta_{w}$ which was introduced in \cite{Wu2004}.

\subsubsection{Related literature on long time behavior of Feynman-Kac semigroups} Non-killed Feynman-Kac semigroups have been widely studied in the literature and we refer for instance to~\cite{daubechies1983one,carmona1990relativistic,kulczycki2006intrinsic,chen2000intrinsic,rousset2006control,kaleta2010intrinsic,guneysu2011feynman,chen2015intrinsic,chen2016intrinsic,ascione2024bulk} for the study of these semigroups in $L^p$ spaces, see also~\cite{Yosida,simon2015quantum,chung2001brownian,bogdan-book}.
We also mention \cite{ferre2} for a very recent investigation of the long time behavior of non-killed Feynman-Kac  semigroups and  its numerical approximations, see also \cite{collet2024branching}. Impossible not to refer to \cite{cat1,cat2} for pioneering works in the case $V=0$ using ultracontractivity, which is linked in the reversible case to an adapted version of assumption {\bf (C3)}. We also refer to~\cite{champagnat2017general,chazottes2016sharp,del2023stability} and references therein for general conditions for ergodicity of non conservative Markov semigroups, see also the classical textbook~\cite{MT1993}.
We finally refer to the recent work  \cite{guillinFK} where   we study the basic properties and   the long time behavior of killed Feynman-Kac semigroups of several models,  arising from statistical physics, with very general singular Schr\"odinger potentials. As already mentioned above, the main goal of this work is to derive a L.D.P. for $\mathbb P_\nu[L_t\in \cdot\ | t<\sigma_{\mathscr D}]$, this is the purpose of the next section.

\subsubsection{Large deviations}

The space $\mathcal P(\mathscr S)$ of probability measures on $\mathscr S$   equipped  with the weak convergence topology   is a Polish space, whose Borel $\sigma$-field is denoted by $\mathcal B(\mathcal P(\mathscr S))$. We say that a subset $B$ of $\mathcal P(\mathscr S)$ is measurable if $B\in \mathcal B(\mathcal P(\mathscr S))$. The weak convergence topology is written
$$\beta_n\xrightarrow{w}\beta.$$
Notice that for any bounded measurable $V:\mathscr S\to \mathbb{R}$, the functional $\beta \in \mathcal P(\mathscr S) \to \beta(V)$   is  $\mathcal B(\mathcal P(\mathscr S))$-measurable, by the regularity of probability measures on the Polish space $\mathscr S$.
The empirical distribution $L_t$ given by \eqref{empirical-L} is a random variable valued in $\mathcal P(\mathscr S)$. A probability measure $\beta\in \mathcal P(\mathscr D)$ is identified with the probability measure on $\mathscr S$ with $\beta(\mathscr D^c)=0$ (i.e. $ \mathcal P(\mathscr D) \equiv \{\beta\in \mathcal P(\mathscr S);\ \beta(\mathscr D^c)=0\}=\{\beta\in \mathcal P(\mathscr S);\ \beta(\mathscr D)=1\}$), and a function on $\mathscr D$ is identified with the function $\mathbf 1_{\mathscr D} f$ on $\mathscr S$.
We consider also on $\mathcal P(\mathscr S)$   the  $\tau$-topology $\sigma(\mathcal P(\mathscr S), b\mathcal B(\mathscr S))$, i.e. the weakest topology such that $\beta\mapsto \beta(f)$ is continuous for each $f\in b\mathcal B(\mathscr S)$, which is stronger than the weak convergence topology. The $\tau$-convergence is written
$$\beta_n\xrightarrow{\tau}\beta.$$

The main result of this work is the following.

\bthm\label{thm-main2}  Assume {\rm {\bf (C1)}}$\to$ {\rm {\bf (C5)}}, and let $\nu \in  \mathcal P_{\mathbf W}(\mathscr D)$. Then:

\begin{enumerate}
\item[{\rm \textbf A}.]
Conditioned to be inside  the set $\mathscr  D$,  $\mathbb P_\nu[L_T\in\cdot |T<\sigma_{\mathscr D}]$, as $T$ goes to infinity, satisfies the  L.D.P.  on $\mathcal P(\mathscr D)$ w.r.t.  the  $\tau$-convergence topology, with speed $T$ and with the rate function
\begin{equation}\label{thm-main2a}
I_{\mathscr D}(\beta) = \sup_{V\in b\mathcal B(\mathscr D)}\big\{\beta(V)-(\Lambda_{\mathscr D}(V) -\Lambda_{\mathscr D}(0))\big\}, \ \beta\in \mathcal P(\mathscr D).
\end{equation}
More precisely:
\begin{enumerate}[(i)]
  \item[\textbf a.]  The rate function $I_{\mathscr D}$ is good or inf-compact, i.e. the level set $\{I_{\mathscr D}\leqslant L\}$ of $I_{\mathscr D}$ is compact in $(\mathcal P(\mathscr D), \tau)$ for any constant $L\in \mathbb{R}^+$;
  \item[\textbf b.] For any open measurable subset $\mathbf G$ of  $(\mathcal P(\mathscr S), \tau)$,
  \begin{equation}\label{thm-main2b}
  \liminf_{T\to \infty} \frac{1}{T} \log \mathbb P_{\nu}\big  [L_T\in \mathbf G|T<\sigma_{\mathscr D}\big ]\ge   -\inf_{\beta\in \mathbf G} I_{\mathscr D}(\beta);
  \end{equation}
  \item[\textbf c.] For any closed measurable subset $\mathbf F$ of  $(\mathcal P(\mathscr D), \xrightarrow{\tau})$ ,
  \begin{equation}\label{thm-main2c}
  \limsup_{T\to \infty} \frac{1}{T} \log \mathbb P_{\nu}\big  [L_T\in \mathbf F |T<\sigma_{\mathscr D}\big ]\leqslant -\inf_{\beta\in \mathbf F} I_{\mathscr D}(\beta);
  \end{equation}
\end{enumerate}

\item [{\rm \textbf B}.] Furthermore
\begin{equation}\label{thm-main2d}
I_{\mathscr D}(\beta)=0 \Longleftrightarrow \beta=\varphi_{\mathscr D} \mu_{\mathscr D}
\end{equation}
where $\mu_{\mathscr D}=\mu_{\mathscr D,0}$, $\varphi_{\mathscr D}=\varphi_{\mathscr D,0}$ is the right positive eigenfunction of the killed Dirichlet semigroup $P_t^{\mathscr D}$ so that $\mu_{\mathscr D}(\varphi_{\mathscr D})=1$. We call $\pi_{\mathscr D}:= \varphi_{\mathscr D} \mu_{\mathscr D}$ the  q.e.d.   of the process in  $\mathscr  D$.

\end{enumerate}
\nthm

Recall that in the second item in Theorem \ref{thm-main2}, $\mu_{\mathscr D}$ is the q.s.d.   of the Markov process in    $\mathscr D$.
We give in   Section \ref{pr.Ex} several examples of processes arising from statistical physics  satisfying {\rm {\bf (C1)}}$\to$ {\rm {\bf (C5)}}.
We then have the following corollary of Theorem \ref{thm-main2}.

\begin{cor}\label{co1}
Assume {\rm {\bf (C1)}}$\to$ {\rm {\bf (C5)}}, and let $\nu \in  \mathcal P_{\mathbf W}(\mathscr D)$.
 For any measurable $\tau$-neighborhood $\mathcal{N}$ of the q.e.d. $\pi_{\mathscr D}:=\varphi_{\mathscr D} \mu_{\mathscr D}$, there are constants $C, \delta>0$ such that
\begin{equation}
 \mathbb P_{\nu} [L_T\notin \mathcal{N}|T<\sigma_{\mathscr D}] \leqslant C e^{-\delta T},\ \forall T>0.
\end{equation}
\end{cor}

In other words, it is a conditional law of large numbers: conditioned   not to leave $\mathscr D$  before time $T$ (i.e. knowing the event $\{T<\sigma_{\mathscr D}\}$), the empirical distribution $L_T$ converges to  $\pi_{\mathscr D}$ (in the $\tau$-topology) exponentially fast in probability as $T\to +\infty$.

\subsubsection{Donsker-Varadhan entropy functional}
We conjecture that the rate function $I_{\mathscr D}(\beta)$ in the L.D.P. above should be
\begin{equation}\label{eq2.16}
I_{\mathscr D}(\beta) = J(\beta) - \inf_{\beta \in \mathcal P(\mathscr S): \ \beta(\mathscr D)=1} J(\beta)
\end{equation}
for all $\beta\in \mathcal P(\mathscr D)$,
where $J(\beta)$ is the Donsker-Varadhan entropy functional  given by
\begin{equation}\label{J_DV}
J(\beta)=\inf_{\mathbb Q} H_{\mathcal F_{[0,1]}}(\mathbb Q|\mathbb P_{\beta}),\ \forall \beta\in \mathcal P(\mathscr S).
\end{equation}
Here the infimum is taken over all probability distributions $\mathbb Q$ on the path space $\mathbb D(\mathbb R_+, \mathscr S)$ so that $X_t(\omega)=\omega(t), \ \omega=(t\mapsto\omega(t))\in \mathbb D(\mathbb R_+, \mathscr S)$ is a stationary stochastic processes with the marginal distribution $\mathbb Q_0(\cdot)=\mathbb Q(X_0\in\cdot)=\beta$ (given); and $\mathcal F_{[0,1]}=\sigma(\omega(t);\ t\in [0,1])$; and for the sub $\sigma$-field $\mathcal G$, $\mathbb P|_{\mathcal G}$ is the restriction on $\mathcal G$ of probability measure $\mathbb P$, and
$$
H_\mathcal G(\mathbb Q|\mathbb P):= \begin{cases}
                   \int \log \frac{d\mathbb Q|_{\mathcal G}}{d\mathbb P|_{\mathcal G}} d\mathbb Q, & \mbox{if } \mathbb Q|_{\mathcal G}\ll \mathbb P|_{\mathcal G} \\
                   +\infty, & \mbox{otherwise}
                 \end{cases}
$$
is the usual relative entropy of $\mathbb Q|_{\mathcal G}$ w.r.t. $\mathbb P|_{\mathcal G}$. In other words $H_{\mathcal F_{[0,1]}}(\mathbb Q|\mathbb P_{\beta})$ is  the relative entropy {\it per unit time} of the stationary process law $\mathbb Q$ w.r.t. our Markov process law $\mathbb P_\beta$ with the same initial distribution $\beta$, and it is the rate function for the level-3 (or path level) large deviations of the Markov process. See \cite{DS89} or \cite{Wu2001} for other variational expressions of $J(\beta)$.

In the general setting of Theorem \ref{thm-main2}, we can only prove one half of \eqref{eq2.16}: $I_{\mathscr D}(\beta) \ge   J(\beta) +\Lambda_{\mathscr D}(0)$, see \eqref{Eq4.5}.
However we can prove it in the reversible case, see indeed the next section.

\subsubsection{The reversible case}
In this section we assume in addition that:
\begin{enumerate}
\item [\textbf{(C6)}] The semigroup $(P_t,t\ge 0)$ is symmetric on $L^2(\mathscr S,\pi)$ where $\pi$ is the unique invariant probability measure of our Markov process $(X_t,t\ge 0)$, namely for all $t\ge 0$:
$$
\<P_t f, g\>_\pi=\<f, P_t g\>_\pi:=\int_{\mathscr S} fP_tg d\pi,\ \forall f,g\in L^2(\mathscr S,\pi),
$$
i.e. $(X_t,t\ge 0)$ is reversible under $\mathbb P_\pi$.
 \end{enumerate}

\bcor\label{cor22} Assume {\bf (C1)} to {\bf (C6)}. Assume also  that $(P_t,t\ge 0)$    is  topologically irreducible\footnote{I.e. for some $t_0>0$ and all $t\ge   t_0$,  $P_t(x,O)>0$ for all $x\in \mathscr S$ and all non-empty subsets $O$ of $\mathscr S$.}  on $\mathscr S$.
 Then, the rate function $I_{\mathscr D}$ given by \eqref{thm-main2a} has the following expression: for any $\beta\in \mathcal P(\mathscr D)$,
\begin{equation}\label{cor22a}
I_{\mathscr D}(\beta)  + \lambda_{\mathscr D}= \begin{cases}
            \mathcal E(\sqrt{h},\sqrt{h}), & \mbox{if } \beta=h \pi\ll  \pi \\
            +\infty, & \mbox{otherwise},
          \end{cases}
\end{equation}
 where
$$
\mathcal E(f,f)= \begin{cases}
            \<\sqrt{-\mathfrak L} f, \sqrt{-\mathfrak L} f\>_{\pi}, & \mbox{if $f$ belongs to  {$\mathbb D_{L^2(\mathscr S, \pi)}(\sqrt{-\mathfrak L})$}}, \\
            +\infty, & \mbox{otherwise},
          \end{cases}
$$
is the Dirichlet form {(where $\mathbb D_{L^2(\mathscr S, \pi)}(\sqrt{-\mathfrak L})$  is the domain of $\sqrt{-\mathfrak L}$ in $L^2(\mathscr S, \pi)$)} and
$$
\lambda_{\mathscr D}:=-\Lambda_{\mathscr D}(0)= \inf\Big \{\mathcal E(f,f); \ \int_{\mathscr S} f^2 d\pi=1, f\mathbf 1_{\mathscr D^c}=0, \ \pi\text{-a.e}. \Big \}
$$
is the Dirichlet eigenvalue.
\ncor

Note indeed  the right hand side of \eqref{cor22a} is the closed expression of the Donsker-Varadhan entropy $J(\beta)$ in the reversible case, and $\lambda_{\mathscr D}=\inf_{\{\beta\in \mathcal P(\mathscr S), \beta(\mathscr D)=1\}} J(\beta)$.
\medskip

The large deviation principle for killed Feynman-Kac semigroups of symmetric Markov processes is an open topic\footnote{It could definitely have attracted Patrick Cattiaux, but he was distracted too often by part of the authors to have time to do so.}. It has also been   considered  recently  in  \cite{kim2024quasi}, with additive functional taking into account the jumps of the process, see also \cite{breyer1999quasi,zhang2014quasi,he2019some,kim2016large,chen2020large} and references therein. 
 We also mention \cite{champagnat2019probabilistic} for   Feynman-Kac   representation formulas and   L.D.P.   of solutions to  deterministic models of phenotypic
adaptation in the small mutations
and large time regime.
 \medskip

\noindent
\textbf{Running example}. Assumptions   {\bf (C1)} to {\bf (C6)} are satisfied for the process solution to \eqref{eq.Lsur} when  \eqref{eq.L1} holds and $\mathbf c=-\nabla U$, for some $U:\mathbb R^d\to \mathbb R_+$.
 In this case, the invariant probability measure $\pi$ is the so-called Gibbs measure $e^{-2U(x)}dx/  Z$, where $  Z=\int_{\mathbb R^d} e^{-2U(y)}dy$.

\section{Proof of Theorem \ref{thm-main1}}
\label{sec.Pr}

\begin{proof}[Proof of Theorem \ref{thm-main1}]
We mention that the proof Theorem \ref{thm-main1} is slightly different from the one we performed in~\cite{guillinFK} as here $V$ is not assumed to be continuous.  
Assume  {\rm {\bf (C1)}}$\to$ {\rm {\bf (C5)}} and {\rm {\bf ($\mathbf{H_V}$)}}.
The proof  of Theorem \ref{thm-main1} is divided into several steps.  
\medskip

\noindent
{\bf Step 1.} In this step we prove that the essential spectral radius of $P_t^{\mathscr D,V}$ on $b_{{\mathbf W}}\mathcal B(\mathscr D)$ is $0$.
To this end, we use, as in \cite[Theorem 3.5 and Theorem 4.1]{guillinqsd}, the non-compact parameters $\beta_w$ and $\beta_{\tau}$ (introduced in \cite{Wu2004}) of a positive bounded kernel $Q(x,dy)$   over $ \mathscr S$
\begin{equation}\label{Eq31}
  \beta_{w}(Q):=\inf_{K\subset\subset \mathscr S}\sup_{x\in \mathscr S}Q(x, K^c) \text{ and }
   \beta_{\tau}(Q):=\sup_{(A_n)}\lim_{n\to \infty} \sup_{x\in \mathscr S}Q(x, A_n)
\end{equation}
where $K\subset\subset \mathscr S$ means that $K$ is a compact subset of $\mathscr S$ and where the supremum above is   is taken over all sequences $(A_n)\subset \mathcal B(\mathscr S)$ decreasing
to $\emptyset$.

For any $t>0$ fixed, set
$$Q_0(x,dy):=\frac{{\mathbf W}(y)}{{\mathbf W}(x)} P_t^{\mathscr D}(x,dy) \ \text{ and } \ Q(x,dy)= \frac{{\mathbf W}(y)}{{\mathbf W}(x)} P_t^{\mathscr D,V}(x,dy).$$  By \cite[proof of Theorem 3.5]{guillinqsd},
$$
\beta_{\tau}(\mathbf 1_K Q_0)=0, \ \forall K\subset\subset \mathscr S.
$$
Since $Q\leqslant e^{t\sup_x |V|(x)} Q_0$, we have also
$$
\beta_{\tau} (\mathbf 1_{K} Q) = 0, \ \forall K\subset\subset \mathscr S.
$$
Hence by \cite[Theorem 3.5]{Wu2004}, the essential spectral radius $\mathsf r_{\ess}(Q|_{b\mathcal B})$ of $Q$ on $b\mathcal B$ (see \cite{Wu2004} for definition), is given by
 \begin{equation}\label{Eq32}
 \mathsf r_{\ess}(Q|_{b\mathcal B}) = \inf_{n\ge   1} [\beta_w(Q^n)]^{1/n}.
 \end{equation}
By \cite[Theorem 3.5]{guillinqsd}, $\beta_w(Q_0)=0$ and therefore
$$
\beta_w(Q)\leqslant e^{t \sup_x V(x)} \beta_w(Q_0)=0.
$$
That implies, by \eqref{Eq32},
$$
\mathsf r_{\ess}(P_t^{\mathscr D,V}|_{b_{{\mathbf W}}\mathcal B})= \mathsf r_{\ess}(Q|_{b\mathcal B}) \leqslant \beta_w(Q) =0,
$$
which is the desired result.
\medskip

\noindent
{\bf Step 2.} In this step we prove that  $P_t^{\mathscr D,V}$ is strongly Feller.
To this end, we show that   for any $f\in b\mathcal B(\mathscr D)$, $P_t^{\mathscr D,V}f$ is continuous on $\mathscr D$.
For any $\vep\in (0,t)$ and for any $f\in b\mathcal B(\mathscr D)$, consider for $x\in \mathscr D$,
$$
Q_\vep f(x):= \mathbb E_x \Big[\mathbf 1_{t<\sigma_{\mathscr D}} f(X_t) e^{\int_{\vep}^{t} V(X_s) ds}  \Big] = P_\vep^{\mathscr D} (P_{t-\vep}^{\mathscr D,V}f)(x).
$$
By \cite[Lemma 5.2]{guillinqsd}, $P_\vep^{\mathscr D}$ is strong Feller, and thus, the function   $Q_\vep f = P_\vep^{\mathscr D} (P_{t-\vep}^{\mathscr D,V}f)$ is continuous. Since on the other hand, it holds:
$$
\aligned
&\sup_{x\in \mathscr D} |P_t^{\mathscr D,V} f(x) - Q_{\vep} f(x)|\\
 &=\sup_{x\in \mathscr  D} \Big |\mathbb E_x \Big [ \mathbf 1_{t<\sigma_{\mathscr D}} f(X_t) e^{\int_{\vep}^{t} V(X_s) ds}\big( e^{\int_0^{\vep} V(X_s) ds} -1\big)  \Big] \Big|\\
&\leqslant \left(e^{\vep \|V\|}-1\right) \|f\| e^{t\|V\|}\\
\endaligned
$$
which goes to zero as $\vep\to0^+$. Hence,  we conclude that the function $P_t^{\mathscr D,V}f$ is continuous on $\mathscr D$.

\medskip

\noindent
{\bf Step 3.} We now conclude the proof of Theorem \ref{thm-main1}.  By the generalized Perron-Frobenius type theorem   \cite[Theorem 4.1]{guillinqsd}, there is a unique couple $(\mu, \varphi)$ where $\mu\in \mathcal P_{\mathbf W}(\mathscr D)$ is a probability measure on $\mathscr D$, charging all non-empty open subsets of $\mathscr D$, and $\varphi \in b_{\mathbf W}\mathcal B (\mathscr D)$ is a continuous and everywhere positive function on $\mathscr D$ with $\mu(\varphi)=1$ such that  
$$
\mu P_1^{\mathscr D,V} = e^{\Lambda } \mu, \  P_1^{\mathscr D,V} \varphi = e^{\Lambda} \varphi,  
$$
where $\Lambda$ is the spectral radius of $P_1^{\mathscr D,V} $ and
\begin{equation}\label{Eq33}
\| e^{-\Lambda n} P_n^{\mathscr D,V}f- \varphi \mu(f)\|_{\mathbf W} \leqslant C e^{-\delta n} \|f\|_{\mathbf W},\ \forall f\in b_{\mathbf W}\mathcal B(\mathscr D),\ n\in \mathbb{N}.
\end{equation}
From this exponential convergence and using the semigroup property, it is quite easy to extend it to whole semigroup $(P_t^{\mathscr D,V},t\ge 0)$ (for the details, see \cite[proof of Theorem 5.3]{guillinqsd}), to finally deduce that all the assertions of Theorem \ref{thm-main1} hold (with $\Lambda=\Lambda_{\mathscr D}(V)$) except the inequality  $\Lambda_{\mathscr D}(V)< \sup_{\mathscr D} V$ which remains to be proved.

Note that
 $\Lambda_{\mathscr D}(V)\le \sup_{\mathscr D} V$ follows from the equality
$$\mathbb E_\mu \Big[  e^{\int_{0}^{t} V(X_s) dx}\mathbf 1_{t<\sigma_{\mathscr D}} \Big]=\mu  P_t^{\mathscr D,V}(\mathbf 1)= e^{\Lambda_{\mathscr D}(V) t} \mu(\mathbf 1), \forall t\ge 0.$$
In addition, the fact that  $\Lambda_{\mathscr D}(V)<\sup_{\mathscr D} V$ follows from the last condition in \textbf{(C5)}. Indeed, if in contrary $\Lambda_{\mathscr D}(V)=\sup_{\mathscr D} V$, then for all $t\ge 0$,
$$\mathbb E_\mu  [  e^{\int_{0}^{t} V(X_s) dx}\mathbf 1_{t<\sigma_{\mathscr D}} ]= e^{t\sup_{\mathscr D} V}.$$
Since $\mu$ charges all non empty open subset of $\mathscr D$ and because the function $x\mapsto e^{\Lambda_{\mathscr D}(V) t}  \mathbf 1- P_t^{\mathscr D,V}(\mathbf 1)$ is non negative and continuous, we get that  for all $x\in \mathscr D$ and $t\ge 0$,
$$ e^{t\sup_{\mathscr D} V}=\mathbb E_x  [  e^{\int_{0}^{t} V(X_s) dx}\mathbf 1_{t<\sigma_{\mathscr D}} ]\le e^{t\sup_{\mathscr D} V}\mathbb P_x  [t<\sigma_{\mathscr D}],$$
so that $\mathbb P_x  [t<\sigma_{\mathscr D}]=1$, i.e.  $\mathbb P_x  [ \sigma_{\mathscr D}=+\infty]=1$, a contradiction with the last condition in~\textbf{(C5)}. Hence, it holds $\Lambda_{\mathscr D}(V)<\sup_{\mathscr D} V$.
\end{proof}

\section{Proof of Theorem \ref{thm-main2}}

\label{secPr2}
\subsection{A generalized G\"artner-Ellis theorem}
Our main tool in the proof of Theorem \ref{thm-main2} is the following well known generalized G\"artner-Ellis theorem (\cite{DS89, Wu97}) that we recall.

\begin{thm} \label{thm_gGE} Let $(\mathbb P_T)_{T>0}$ be a sequence of probability distribution on $\mathcal P(\mathscr S)$, such that:

\begin{enumerate}
\item For any potential $V\in b\mathcal B(\mathscr D)$,
\begin{equation}
\boldsymbol \Lambda(V)=\lim_{T\to +\infty} \frac{1}{T} \log \int_{\mathcal P(\mathscr D)} e^{T\beta(V)} \mathbb P_T (d\beta).
\end{equation}

  \item The mapping $V\mapsto \boldsymbol \Lambda(V)$ is Gateaux differentiable on $b\mathcal B(\mathscr D)$, i.e. the mapping $t\mapsto \boldsymbol \Lambda(V+t V_1)$ is differentiable  at $t=0$ for any potential $V,V_1\in b\mathcal B(\mathscr D)$.
  \item $\mathbb P_T$ is exponentially $*$-tight, namely,  $\forall L>0,\ \exists \mathbf K $ compact in $(\mathcal P(\mathscr D),\tau)$ such that for any measurable neighborhood $\mathcal{N}$ of $\mathbf K$,
  $$
  \limsup_{T\to +\infty} \frac{1}{T} \log \mathbb P_T [L_T\notin \mathcal{N}] \leqslant -L.
  $$
\end{enumerate}
Then, $\mathbb P_T$ satisfies the {\rm L.D.P.} on $(\mathcal P(\mathscr D), \tau)$ with speed $T$ and the rate function
$$
I(\beta)=\sup\left\{\beta(V)-\boldsymbol \Lambda (V),\ V\in \mathcal C_b(\mathscr D)\right\}, \ \beta\in \mathcal P(\mathscr D),
$$
which is the Legendre transform $\boldsymbol  \Lambda^*$ of the Cramer functional $\boldsymbol \Lambda$.
\end{thm}

\brmk\label{rem42} {\rm In the theorem above, without the exponential $*$-tightness condition (3), the local L.D.P. below holds: for any $\beta\in \mathcal P(\mathscr D)$,

\begin{enumerate}
  \item  for any measurable neighborhood $\mathcal{N}$ of $\beta$ (in $\mathcal P(\mathscr D)$),
  \begin{equation}\label{rem42a}
\lim\inf_{T\to +\infty}
                       \frac{1}{T}\log \mathbb P_T(\mathcal{N}) \ge   - I(\beta);
   \end{equation}

  \item for any $a<I(\beta)$, there is a measurable neighborhood $\mathcal{N}$ of $\beta$,
  \begin{equation}\label{rem42b}
  \lim_{T\to +\infty}  \frac{1}{T}\log \mathbb P_T(\mathcal{N}) \leqslant - a.
\end{equation}
\end{enumerate}
The exponential $*$-tightness is well-adapted to the non-metrisable topology (such as the $\tau$-topology), and it is equivalent to the usual exponential tightness for large deviations of a sequence of probability distributions on Polish spaces (see \cite{Wu97}).
}\nrmk

\begin{proof}[Proof of Theorem \ref{thm-main2}]
Assume  {\rm {\bf (C1)}}$\to$ {\rm {\bf (C5)}}.
The proof is divided into several steps and consists in using Theorem \ref{thm_gGE} above.

\medskip

\noindent
{\bf Step 1.} In this step we prove the equality \eqref{eq.Ll}. Let
$$\mathbb P_T(\cdot)=\mathbb P_\nu\big [L_T\in \cdot|T<\sigma_{\mathscr D}\big ].$$ Let $\Lambda_{\mathscr D}(V)$ be the minus Dirichlet eigenvalue given by \eqref{sp_radius}, $\mu_{\mathscr D,V}$ the left eigen-probability distribution,
$\varphi_{\mathscr D,V}$ the right eigenfunction, of $P_t^{\mathscr D,V}$ over $b_{\mathbf W}\mathcal B(\mathscr D)$ which are given in Theorem \ref{thm-main1}. We have  by Theorem \ref{thm-main1},
$$
\aligned
\int_{\mathcal P(\mathscr S)} e^{T\beta(V)} \mathbb P_T (d\beta)&=  \frac{\mathbb E_\nu\big[e^{\int_{0}^{T} V(X_t) dt} \mathbf 1_{T<\sigma_{\mathscr D}}\big]}{\mathbb P_\nu[T<\sigma_{\mathscr D}]}\\
&= e^{\left(\Lambda_{\mathscr D}(V)-\Lambda_{\mathscr D}(0)\right)T} \, \frac{e^{-\Lambda_{\mathscr D}(V)T} \nu (P_T^{\mathscr D,V} \mathbf 1 )}{\nu(e^{-\Lambda_{\mathscr D}(0) T}  P_T^{\mathscr D} \mathbf 1)} \\
&= e^{\left(\Lambda_{\mathscr D}(V)-\Lambda_{\mathscr D}(0)\right)T}\,  \frac{\nu(\varphi_{\mathscr D, V}) \mu_{\mathscr D,V}(\mathbf 1) + O(e^{-\delta T})}{ \nu(\varphi_{\mathscr D,0}) \mu_{\mathscr D}(\mathbf 1) + O(e^{-\delta T})}
\endaligned
$$
Thus, one has that $\boldsymbol \Lambda(V)$ is given by:
\begin{equation}\label{eq.Ll}
\boldsymbol \Lambda(V)=\lim_{T\to +\infty} \frac{1}{T} \log \int_{\mathcal P(\mathscr S)} e^{T\beta(V)} \mathbb P_T (d\beta)=\Lambda_{\mathscr D}(V)-\Lambda_{\mathscr D}(0).
\end{equation}

\noindent
{\bf Step 2 (Gateaux-differentiability).} By Theorem \ref{thm-main1},  we recall that the operator $P_1^{\mathscr D,V}$ has a spectral gap near its spectral radius $e^{\Lambda(V)}$ on $b_{{\mathbf W}}\mathcal B(\mathscr D)$ and its eigen-projection is one-dimensional. Given two potentials $V,V_1\in b\mathcal B(\mathscr D)$, since for any $f\in b_{\mathbf W}\mathcal B(\mathscr D)$, the mapping
$$\lambda \in \mathbb C \mapsto P_1^{\mathscr D, V+\lambda V_1}f$$
valued in the Banach space $b_{{\mathbf W}}\mathcal B(\mathscr D)$,  is analytic on a neighborhood of $0$ in $\mathbb C$, we deduce by the perturbation theory of operators \cite[Chapter 7, Theorems 1.7 and 1.8]{Kato}), that the mapping $$\lambda \mapsto e^{\Lambda_{\mathscr D}(V+\lambda V_1)} \text{ (the spectral radius of $P_1^{ \mathscr D,V+\lambda V_1}$)}$$   is also analytic on a neighborhood of $0$ in  $\mathbb{C}$. By \eqref{eq.Ll}, this shows item (2) in Theorem~\ref{thm_gGE}, namely that $V\in b\mathcal B(\mathscr D) \mapsto \boldsymbol \Lambda(V)$ is Gateaux differentiable.
\medskip

Let us recall  that those first two steps yield the local L.D.P. for $\mathbb P_\nu[L_T\in\cdot|T<\sigma_{\mathscr D}]$ with the rate function $I_{\mathscr D}$, see indeed Remark \ref{rem42}.

\medskip

\noindent
{\bf Step 3 (exponential *tightness).}  We now prove item (3) in Theorem~\ref{thm_gGE}.
Let us first recall that it follows from \cite{Wu2001}, that under the assumptions {\bf (C1)}, {\bf(C2)}, {\bf (C3)},  $\mathbb P_\nu[L_T\in \cdot]$ satisfies, on $(\mathcal P(\mathscr S), \tau)$, the good upper bound of large deviations with Donsker-Varadhan's rate function $J(\beta)$ given by \eqref{J_DV}. Thus for any measurable $\tau$-closed subset $\mathbf F$ of $\mathcal P(\mathscr D)$, $$\mathbf F=\mathbf F\cap \{\beta\in \mathcal P(\mathscr S); \beta(\mathscr D)=1\} \text{  is closed in $(\mathcal P(\mathscr S), \tau)$,}$$
 and
$$\{L_T\in \mathbf F,\ T<\sigma_{\mathscr D}\}\subset \big\{L_T\in \mathbf F\cap \{\beta\in \mathcal P(\mathscr S); \beta(\mathscr D)=1\}\big\},$$
 and thus we get the following upperbound:
$$
\limsup_{T\to+\infty} \frac{1}{T}\log \mathbb P_\nu\big [L_T\in \mathbf F, T<\sigma_{\mathscr D}\big ]\leqslant - \inf_{\beta\in \mathbf F} J(\beta).
$$
Since by Theorem \ref{thm-main1} (with $V=0$), we have
$$
\lim_{T\to+\infty} \frac{1}{T}\log \mathbb P_\nu[T<\sigma_{\mathscr D}]=\Lambda_{\mathscr D}(0),
$$
we thus finally deduce that
\begin{equation}\label{Eq4.4}
\limsup_{T\to+\infty} \frac{1}{T}\log \mathbb P_\nu\big [L_T\in \mathbf F|T<\sigma_{\mathscr D}\big ]\leqslant - \inf_{\beta\in \mathbf F} J(\beta) - \Lambda_{\mathscr D}(0).
\end{equation}
This upper bound, together with the local L.D.P. remarked above, yields the following inequality:
\begin{equation}\label{Eq4.5}
I_{\mathscr D}(\beta) \ge   J(\beta) + \Lambda_{\mathscr D}(0)=J(\beta) -\lambda_{\mathscr D},\ \forall \beta\in \mathcal P(\mathscr D).
\end{equation}
Note that this   is half of the     equality \eqref{eq2.16}.

Now for any $L>0$, consider  the set $\mathbf K$ defined by
$$\mathbf K:=\big \{\beta\in \mathcal P(\mathscr S); \beta(\mathscr D)=1,\ J(\beta)\leqslant L-\Lambda_{\mathscr D}(0)\big \}\subset \mathcal P(\mathscr D).$$
 It is a compact subset of  $(\mathcal P(\mathscr D), \tau)$.
We then obtain  using  the upper bound of large deviations above that for any measurable $\tau$-open neighborhood $\mathcal{N}$ of $\mathbf K$,
$$
\limsup_{T\to+\infty} \frac{1}{T}\log \mathbb P_\nu [L_T\notin \mathcal{N}|T<\sigma_{\mathscr D}]\leqslant - \inf_{\beta\notin \mathcal{N}} J(\beta) - \Lambda_{\mathscr D}(0)\leqslant -L.
$$
This is the desired expotential $*$-tightness.

\medskip

\noindent
{\bf Step 4. } To finish the proof, it remains to show \eqref{thm-main2d}. Notice that $I(\beta)=0$ if and only if $\beta\in \partial \Lambda_{\mathscr D}(0)$ (the sub-differential of $\Lambda_{\mathscr D}$ at $V=0$). Since $\Lambda_{\mathscr D}$ is Gateaux-differentiable, the level set $\{I=0\}$ is a singleton. Thus it remains to prove that $\pi_{\mathscr D}=\varphi_{\mathscr D}\mu_{\mathscr D}\in \partial \Lambda_{\mathscr D}(0)$ or equivalently that
\begin{equation}\label{eq4.7}
\Lambda_{\mathscr D}(V)\ge   \pi_{\mathscr D}(V), \ \forall V\in b\mathcal B(\mathscr D).
\end{equation}
To this end, we have at first by Jensen's inequality:
$$
\aligned
\log \mathbb E_\nu\Big[  e^{\int_{0}^{T} V(X_t) dt} |T<\sigma_{\mathscr D}\Big] &\ge   \mathbb E_\nu\Big[\int_{0}^{T} V(X_t) dt |T<\sigma_{\mathscr D}\Big] \\
&= \frac{\mathbb E_\nu  \Big[ \int_{0}^{T}  \mathbf 1_{t<\sigma_{\mathscr D}} V(X_t) \mathbf 1_{T<\sigma_{\mathscr D}} dt \Big]}{\mathbb P_\nu [T<\sigma_{\mathscr D}]}\\
&= \frac{   \int_{\mathscr D} \int_{0}^{T} P^{\mathscr D}_t V P^{\mathscr D}_{T-t} \mathbf 1 dt d\nu   }{  \int_{\mathscr D} P_{T}^{\mathscr D}  \mathbf 1  d\nu}.\\
\endaligned
$$
In addition, by Theorem \ref{thm-main1}, we have
$$
e^{-\Lambda_{\mathscr D}(V) T}\int_{\mathscr D} P_{T}^{\mathscr D} \mathbf 1 d\nu= \nu(\varphi_{\mathscr D})\mu_{\mathscr D}( \mathbf 1) + O(e^{-\delta T}) \int_{\mathscr D} {\mathbf W} d\nu= \nu(\varphi_{\mathscr D}) + O(e^{-\delta T}),
$$
and (recall that the potential  $V$ is bounded)
$$
\aligned
&e^{-\Lambda_{\mathscr D}(V) T} \int_{\mathscr D}  \int_0^T P^{\mathscr D}_t( V P^{\mathscr D}_{T-t}1) dt d\nu \\
&= \int_{0}^{T}\int_{\mathscr D} e^{-\Lambda_{\mathscr D}(V) t}  P_t^{\mathscr D} \left(V \varphi_{\mathscr D} + O(e^{-\delta(T-t)})|V|\cdot {\mathbf W}\right) d\nu dt \\
&= \int_{0}^{T}\int_{\mathscr D} \left[\varphi_{\mathscr D} \mu_{\mathscr D}(V \varphi_{\mathscr D}) +  O(e^{-\delta T}) {\mathbf W} \right]d\nu dt \\
&= T \pi_{\mathscr D}(V) \cdot \nu(\varphi_{\mathscr D}) + O(1).
\endaligned
$$
Thus we obtain
$$
\aligned
\Lambda_{\mathscr D}(V)&\ge   \limsup_{T\to +\infty} \frac{1}{T} \log \mathbb E_\nu\Big[e^{\int_{0}^{T} V(X_t) dt} \big |T<\sigma_{\mathscr D}\Big]\\
&\ge    \lim_{T\to +\infty} \frac{1}{T}\,  \frac{ e^{-\Lambda_{\mathscr D}(V) T}\int_{\mathscr D} \int_{0}^{T} P^{\mathscr D}_t V P^{\mathscr D}_{T-t}1 dt d\nu   }{ e^{-\Lambda_{\mathscr D}(V) T}\int_{\mathscr D} P_{T}^{\mathscr D} 1  d\nu}\\
&= \lim_{T\to +\infty} \frac{1}{T}\, \frac{T \pi_{\mathscr D}(V) \cdot \nu(\varphi_{\mathscr D}) + O(1) }{\nu(\varphi_{\mathscr D}) + O(e^{-\delta T})}\\
&= \pi_{\mathscr D}(V).
\endaligned
$$
This is the desired result \eqref{eq4.7}. The proof of Theorem \ref{thm-main2} is thus complete.
\end{proof}

\brmk {\rm A quite natural approach for the large deviations of $\mathbb P_\nu[L_T\in\cdot |T<\sigma_{\mathscr D}]$ is to use Varadhan-Ellis principle by approximation (this was suggested by J.D. Deuschel). Let us consider for $n\ge 1$, the potential functions $V_n$ defined by:
$$
V_n(x) = -n \mathbf 1_{\mathscr D^c}(x). 
$$
The killed semigroup $P_t^{\mathscr D}f(x)$ can be approximated by the (non-killed) Feynman-Kac semigroup
$$
P_t^{V_n} f(x) = {\mathbb E}_x \Big[\exp\left(\int_{0}^{t} V_n(X_s) ds\right)f(X_t)\Big].
$$
Under {\bf (C1)} to {\bf (C3)} together with the topological irreducibility of the process on the whole space $\mathscr S$,  for each $n\ge 1$ fixed, as $V_n$ is bounded, by the L.D.P. in \cite{Wu2001} and Varadhan-Ellis principle, the family 
$$
Q_n(L_T\in \cdot):= \frac{{\mathbb E}_\nu \Big[\mathbf 1_{[L_T\in\cdot]}\exp\left(\int_{0}^{t} V_n(X_s) ds\right)\Big]}{{\mathbb E}_\nu \Big[\exp\left(\int_{0}^{t} V_n(X_s) ds\right)\Big]}
$$
satisfies the L.D.P. on $(\mathcal P(\mathscr S), \tau)$ with the rate function 
$$
I_n(\beta) = J(\beta) - \Lambda(V_n),\ \forall \beta\in \mathcal P(\mathscr S). 
$$
On the one hand, if we now let  $n\to+\infty$, $
Q_n(L_T\in \cdot)$ converges to the target distribution: $ \mathbb P_\nu[L_T\in\cdot |T<\sigma_{\mathscr D}]$ and this is  satisfying. However, on the other hand,  it remains now  to exchange the limit order. For the upper bound of L.D.P., it is enough to prove that $\Lambda(V_n)\to \Lambda(0)$ (which is already quite difficult). But the main difficulty with this  approach  follows from   the lower bound of L.D.P.. Without further assumptions on the subset ${\mathscr D}$, such as the connectedness in the case when the paths of $(X_t,t\ge 0)$ are a.s. continuous, it is easy to see that the L.D.P. in Theorem \ref{thm-main2}  with a convex rate function fails.}
\nrmk

\section{Proof of Corollary \ref{cor22}}
\label{sec.Pr3}

\begin{proof}[Proof of Corollary \ref{cor22}]
 Assume {\rm {\bf (C1)}}$\to$ {\rm {\bf (C6)}}.
The proof of Corollary \ref{cor22} is divided into several steps.
\medskip

\noindent
{\bf Step 1 (Preparation). }
In the framework of {\rm {\bf (C1)}}$\to$ {\rm {\bf (C3)}}, under the extra condition that $P_t$ is topologically irreducible on $\mathscr S$,  our Markov process $(X_t,t\ge 0)$ admits a unique invariant probability measure $\pi$, which charges all non-empty subsets of $\mathscr S$, and $P_{t}(x,dy)=p_t(x,y) \pi(dy)$ is absolutely continuous w.r.t. $\pi$ for all $t\ge   2t_0$, see indeed \cite{Wu2001}. Thus, we have
\begin{equation}\label{Eq5.2}
J(\beta)<+\infty \Longrightarrow \beta\ll  \pi,
\end{equation}
 a fact which was noted in \cite{Wu2001}.  Let us now consider the quadratic form $\mathcal E$ defined by
$$\mathcal E(f,f)=\<\sqrt{-\mathfrak L} f, \sqrt{- \mathfrak L} f\>_{L^2(S,\pi)},$$ with domain $\mathbb D(\mathcal E)=\mathbb D_{L^2(S, \pi)}(\sqrt{-\mathfrak L})$. The quadratic form $\mathcal E$  is  the so-called Dirichlet form of the (non-killed) reversible Markov semigroup $(P_t,t\ge 0)$ over $\mathscr S$.
It is well known (\cite{DS89}) that for $\beta= h\pi \in \mathcal P(\mathscr S)$,
\begin{equation}\label{Eq5.3}
J(\beta)=\begin{cases}
           \mathcal E(\sqrt{h}, \sqrt{h}), & \mbox{if } \sqrt{h}\in \mathbb D(\mathcal E); \\
           +\infty, & \mbox{otherwise}.
         \end{cases}
\end{equation}

%
%

\noindent
{\bf Step 2 (Rayleigh's principle). } Given a potential function $V\in b\mathcal B(\mathscr D)$, the semigroup $P_t^{\mathscr D,V}$ is symmetric on $L^2(\mathscr D, \mathbf 1_{\mathscr D} \pi)$. Its log spectral radius on $L^2(\mathscr D, \mathbf 1_{\mathscr D} \pi)$, defined by
$$
\Lambda_{\mathscr D,2}(V) := \lim_{T\to+\infty} \frac{1}{T}\log \|P_t^{\mathscr D,V}\|_{L^2(\mathscr D, \mathbf 1_{\mathscr D} \pi)},
$$
is always not greater than its log spectral radius on    in the Banach space $b_{\mathbf W}\mathcal B(\mathscr D)$, i.e. than  $\Lambda_{\mathscr D}(V)$ (by the spectral decomposition). On the other hand,
take an initial distribution $\nu=h \pi$ so that $h\mathbf 1_{\mathscr D^c}=0, \ \pi\text{-a.e}.$, $h\in L^2( \mathscr S,\pi)$ and $\nu(\mathbf W)<+\infty$. Then, we have  by Theorem \ref{thm-main1}:
$$
\Lambda_{\mathscr D,2}(V)\ge   \lim_{T\to+\infty} \frac{1}{T}\log \nu(P_t^{\mathscr D,V}\mathbf 1)= \Lambda_{\mathscr D}(V).
$$
Therefore, it holds for all $V\in b\mathcal B(\mathscr D)$:
\begin{equation}\label{Eq5.1}
\Lambda_{\mathscr D}(V)=\Lambda_{\mathscr D,2}(V).
\end{equation}

The quadratic Dirichlet form associated with $P_t^{\mathscr D,V}$ on $L^2(\mathscr D, \mathbf 1_{\mathscr D} \pi)$ is defined by (see indeed \cite{MR92}), for all $f\in \mathbb D(\mathcal E_{\mathscr D,V})=\{g\in \mathbb D(\mathcal E); \ \mathbf 1_{\mathscr D^c}g=0, \pi\text{-a.e}.\}$,
\begin{align*}
\mathcal E_{\mathscr D,V}(f, f) &= \mathcal E(f,f) - \int_{\mathscr D} V f^2 d\pi.
\end{align*}
By Rayleigh's principle, one has:
$$
\aligned
& \Lambda_{\mathscr D,2}(V) \\
&= \sup\{-\mathcal E_{\mathscr D,V}(f, f); \ f\in \mathbb D(\mathcal E_{\mathscr D,V}), \pi(f^2)=1\}\\
&= \sup\left\{ \int_{\mathscr S} V f^2 d\pi -\mathcal E(f,f); \ f\in \mathbb D(\mathcal E), \pi(f^2)=1\text{ and } f\ge   0, \mathbf 1_{\mathscr D^c}f=0, \ \pi\text{-a.e}.\right\}\\
\endaligned
$$
where we have used the fact that $\mathcal E_{\mathscr D,V} (|f|,|f|)\leqslant \mathcal E_{\mathscr D,V} (f,f)$. Thus by \eqref{Eq5.1} and \eqref{Eq5.2}, we  deduce that
$$
\Lambda_{\mathscr D}(V)=\Lambda_{\mathscr D,2}(V)  = \sup\left\{ \int_{\mathscr S} V d\beta-J(\beta); \ \beta\in \mathcal P(\mathscr D)\right\}.
$$
As $\boldsymbol \Lambda(V)=\Lambda_{\mathscr D}(V)-\Lambda_{\mathscr D}(0)$  {(see \eqref{eq.Ll})}  {and since by definition $\Lambda_{\mathscr D}(0) =   -\lambda_{\mathscr D}$}, we get by Legendre-Fenchel's theorem,
$$
I_{\mathscr D}(\beta)=\boldsymbol \Lambda^*(\beta) = (\Lambda_{\mathscr D})^*(\beta) + \Lambda_{\mathscr D}(0) = J(\beta) -\lambda_{\mathscr D}, \ \forall \beta\in \mathcal P(\mathscr D).
$$
This is the desired result. The proof of  Corollary \ref{cor22} is thus complete.
\end{proof}


\section{Examples}
\label{pr.Ex}
In this section we give some examples, arising from statistical physics,  of processes satisfying \textbf{(C1)} $\to$ \textbf{(C5)}.

\subsection{Kinetic Langevin process driven by a Brownian motion}
Let $U:\mathbb R^d\to [1,+\infty]$ be  measurable function
and consider the phase space
$$\mathscr S=\{U<+\infty\}\times \mathbb R^d.$$
Let us consider the so-called  kinetic Langevin process $(X_t=(x_t,v_t),t\ge 0)$ which is   the solution to the stochastic differential equation in $\mathscr S$:

 \begin{equation}\label{eq.Lcin}
dx_t=v_tdt, \ dv_t=-\nabla U(x_t)dt-\gamma v_tdt + dB_t,
\end{equation}
where $\gamma>0$ is the friction parameter and $(B_t,t\ge 0)$ is $d$-dimensional standard Brownian motion. The validity of the conditions \textbf{(C1)} $\to$ \textbf{(C5)} have been shown  in:
\begin{enumerate}
\item[\textbf a.]  In \cite{guillinqsd} when  the potential  $U$ is only $\mathcal C^1$ over $\mathbb R^d$, namely when $\nabla U$ is continuous over $\mathbb R^d$ (in this case $\mathscr S=\mathbb R^d\times \mathbb R^d$), and when   $\mathscr D =\mathscr O\times \mathbb R^d$  where   $\mathscr O$ is a $\mathcal C^2$ subdomain  of $ \mathbb R^d$.
\item[\textbf b.] In  \cite{guillinqsd2} when the potential $U$ models the  singular interactions between the particules and when   $\mathscr D =\mathscr O\times \mathbb R^d$  where   $\mathscr O$ is a subdomain of $\{U<+\infty\}$   with   $\mathcal C^2$ boundary inside $\{U<+\infty\}$.  More precisely, in this case, $d=kN$ (with $k,N\ge 1$), $N$ is the number of $\mathbb R^k$-particles, and  $U$ has the form
$$U(x^1,\ldots,x^N)=\sum_{i=1}^N V_c(x^i)+\sum_{1\le i<j\le N}V_I(x^i-x^j), \ x_i\in \mathbb R^k,$$
where $V_I:\mathbb R^k\to  \mathbb R\cup\{+\infty\}$ is  e.g. the Coulomb potential, the Riesz potential,  or the Lennard-Jones potential.
\end{enumerate}
  More recently, the authors proved in~\cite{guillinqsd3} that in both cases \textbf a and  \textbf b above, the conditions \textbf{(C4)} and  \textbf{(C5)}  actually  hold   with any such subdomain $\mathscr D$ of the phase space $\mathscr S$ (i.e. without any assumption on the regularity of $\mathscr O$).
 %
  We refer to~\cite{ramilarxiv2,benaim2021degenerate,champagnat2024quasi} for related results
%

\subsection{SDE   driven by a rotationally invariant stable processes}

In this section, we prove Theorem \ref{th.LL} below for the process   solution to the elliptic  stochastic differential equation \eqref{eq.Levy} driven by a rotationally invariant stable processes. This theorem aims at showing that such processes satisfy under mild assumptions the conditions {\rm\textbf{(C1)}} $\to$ {\rm\textbf{(C5)}}.

  \subsubsection{Definition of the process and assumptions}
Let  $(\Omega, \mathcal F, (\mathcal F_t)_{t\ge 0}, \mathbb P)$ be a  filtered probability space  (where the filtration satisfies the usual condition).
Let us consider a (L\'evy) rotationally invariant $\alpha$-stable  process $(L^\alpha_t,t\ge 0)$ on $\mathbb R^d$ ($0<\alpha<2, d\ge 1)$, see e.g.~\cite[Example 3.3.8 and Section 4.3.4]{applebaum2009levy}. We denote by $F_\alpha$ its L\'evy measure and we recall that it is pure jump  process where:
  \begin{equation}\label{eq.nu}
F_\alpha(dz)= \frac{C_\alpha  }{|z|^{d+\alpha}}\,  dz, \ C_\alpha>0.
\end{equation}
Recall also  that for all $t\ge 0$,  $L^\alpha_t$ admits moments of order $q\in [0,\alpha)$.
Let $    \beta>1$ and  $U: \mathbb R^d\to [1,+\infty)$ be  a $\mathcal C^2$ function such that for some $R_*>1$ and $c_*>0$,
 \begin{equation}\label{eq.Uh}
 \nabla U(x)\cdot x \ge   c_* |x|^{2    \beta} \text{ for all $|x|>R_*$}.
 \end{equation}
 Remark that in the case $U(x)=|x|^k$ (for large $|x|$), condition \eqref{eq.Uh} is verified if $k>2$ with $\beta=k/2$.
 Let $(X_t(x),t\ge 0)$ be the solution (see Corollary \eqref{co.Ex}) of the L\'evy driven elliptic stochastic differential equation
 \begin{equation}\label{eq.Levy}
dX_t= -\nabla U(X_t)dt + dL^\alpha_t, \ X_0=x\in \mathbb R^d.
\end{equation}
 For a non empty subset  $  \mathscr D$ of $\mathbb R^d$, we recall that $\sigma_{\mathscr D}=\inf\{t\ge 0, X_t\notin  \mathscr D\}$.  
 In the rest of this section, we check that the process $(X_t,t\ge 0)$ satisfies \textbf{(C1)}  $\to$ \textbf{(C5)}.  In what follows, $B(x,r)$ is the open ball of $\mathbb R^d$ centered at $x$ of radius $r>0$.
  Let us mention that one can easily adapt our analysis to non gradient vector field in \eqref{eq.Levy}.

\subsubsection{On  Assumptions {\rm \textbf{(C1)}} and {\rm \textbf{(C3)}}}
The infinitesimal generator of \eqref{eq.Levy} is given by, for $\psi \in \mathcal C^2_c(\mathbb R^d)$ (see e.g.~\cite[Section 6.7]{applebaum2009levy}),
 $$\mathscr L^X\psi(x)= -\nabla U(x)\cdot \nabla \psi (x)+\int_{\mathbb R^d} \big [ \psi(x+z)-\psi(x)-\nabla \psi(x)\cdot z \mathbf 1_{\{|z|\le 1\}}\big] F_\alpha(dz).$$
  In the following
 $\theta>0$ is small enough such that
\begin{equation}\label{eq.a1}
   2\beta \theta < \alpha \wedge 1.
  \end{equation}
  Consider a smooth function  $\mathbf V:\mathbb R^d\to [1,+\infty)$ such that for $|x|>1$, $\mathbf  V(x)=2+ |x|^{    \beta\theta}$.
Then, for $p>1$, define the function  $\mathbf W$ by
\begin{equation}\label{eq.LW-S}
 \mathbf W= \mathbf  V^{1/p}.
\end{equation}

 \begin{prop}\label{pr.C3}
 Assume \eqref{eq.Uh}.
Then, for any $p>1$, {\rm \textbf{(C3)}} is satisfied with  $\mathbf W$ defined by \eqref{eq.LW-S}.
 \end{prop}

 \begin{proof} Recall  that $\beta>1$. Let $x\in \mathbb R^d$. 
If   $|x|>\max(R_*,1)$, we have $\mathbf  V(x)=2+ |x|^{    \beta\theta}$ and  in this case one has $\nabla\mathbf V(x)=    \beta\theta \, x\,  |x|^{    \beta\theta-2}$. Therefore, using  \eqref{eq.Uh}, one has  $ -\nabla U(x)\cdot \nabla \mathbf V (x)\le - c_*   \beta\theta |x|^{    \beta\theta +2    \beta-2}$ and thus, since $2    \beta>2$, one has when  $|x|\to +\infty$,  $- {\nabla U(x)\cdot \nabla \mathbf V (x)}/{\mathbf V (x)}  \to -\infty$.
Let us now consider for $x\in \mathbb R^d$,
$$\mathscr A_1(x):=  \int_{|z|\le 1}  \big [\mathbf V(x+z)-\mathbf V(x)-\nabla\mathbf V(x)\cdot z  \big] F_\alpha (dz).$$
Since $\text{Hess } \mathbf V$ is bounded over $\mathbb R^d$  , it holds for all $x,z\in \mathbb R^d$:
$|\mathbf V(x+z)-\mathbf V(x)-\nabla\mathbf V(x)\cdot z|\le C |z|^2$, for some $C>0$.
  Hence,   one has for all $x\in \mathbb R^d$, $ | \mathscr A_1(x)| \le C \int_{|z|\le 1} |z|^2 F_\alpha (dz)$.
Thus, $ \mathscr A_1(x)$ is well defined and   $  \mathscr A_1(x) /\mathbf V (x) \to 0 \text{ as } |x|\to +\infty$.
Let us now consider  the quantity
\begin{align*}
&\mathscr A_2(x):=  \int_{|z|> 1}  \big [\mathbf V(x+z)-\mathbf V(x) \big] F_\alpha (dz) \\
&=  \int_{|z|> 1}  \big [\mathbf V(x+z)-\mathbf V(x) \big]\mathbf1 _{|x+z|\ge 1 \text{ and } |x|\ge 1}  F_\alpha (dz) \\
&\quad +\int_{|z|> 1}  \big [\mathbf V(x+z)-\mathbf V(x) \big] \mathbf1 _{|x+z|< 1 \text{ or } |x|< 1}  F_\alpha (dz) =: \mathscr A^a_2(x)+ \mathscr A^b_2(x).
 \end{align*}
 Let $|z|\ge 1$ and $x\in \mathbb R^d$.
 Assume that $x\in \mathbb R^d$ is such that $|x+z|\ge 1$ and $|x|\ge 1$ so that $\mathbf V(x+z)=2+|x+z|^{\beta\theta}$ and $\mathbf V(x)=2+|x|^{\beta\theta}$.
Since $    \beta\theta<1$, using the inequality  $(a +b)^{\beta\theta}\le a^{    \beta\theta}+b^{    \beta\theta} $ for $a,b\ge 0$, we have thanks to \eqref{eq.nu} and  \eqref{eq.a1} that  in this case:
$
 \mathscr A_2^a(x)    \le   \int_{|z|> 1}    [(|x|+|z|)^{     \beta\theta}- |x|^{    \beta\theta}]  F_\alpha (dz) \le  \int_{|z|> 1}    |z|^{     \beta\theta} F_\alpha (dz)<+\infty$.
Assume now that $x\in \mathbb R^d$ is such that $|x+z|< 1$ so that $\mathbf V(x+z)\le c_2:=\sup_{|y|\le 1}\mathbf V$. We then have since $\mathbf V\ge 0$:
\begin{align*}
 \mathscr A^b_2(x)  & \le  \int_{|z|> 1}    [c_2- \mathbf V(x)]  F_\alpha (dz) \le c_2 \int_{|z|> 1}  F_\alpha (dz)<+\infty,
\end{align*}
Assume now that $x\in \mathbb R^d$ is such that $|x|< 1$. Since for some $L>0$, $\mathbf V(y)\le L(2+|y|^{\beta\theta})$, we then have  using again that $\mathbf V\ge 0$:
\begin{align*}
 \mathscr A_2^b(x)  & \le L \int_{|z|> 1}    [2+(|x|+|z|)^{\beta\theta}]  F_\alpha (dz) \le  2L\int_{|z|> 1}  \!\!\!\!F_\alpha (dz)+ L\int_{|z|> 1}  \!\!  ( 1+|z|^{\beta\theta})  F_\alpha (dz).
\end{align*}
Consequently $ \mathscr A_2^b(x)<+\infty$.
In all cases, we have that
$ \mathscr A_2(x) /\mathbf V (x) \to 0$ as $|x|\to +\infty$ (note also that $|\mathscr A_2|\le C(\mathbf V+1)$ for some $C>0$).
In conclusion, we have proved that for some $m>0$, $ |\mathscr L^X\mathbf V(x)|\le m( |x|^{    \beta\theta +2    \beta-2}+1)$ and  $  \mathscr L^X\mathbf V(x)/\mathbf V(x)\to -\infty$ as  $|x|\to +\infty$. As $\mathbf V\in \mathbb D_e(\mathfrak L^X)$, where $\mathbb D_e(\mathfrak L^X)$ is  the extended domain of the process \eqref{eq.Levy}, and $\mathfrak L^X\mathbf V=\mathscr L^X\mathbf V$ (see indeed the end of the proof of Corollary~\ref{co.Ex} below), we have thus proved that \textbf{(C3)} holds.
 \end{proof}

 \begin{cor} \label{co.Ex} Assume \eqref{eq.Uh}.
Then, for all $x\in \mathbb R^d$, there exists a unique pathwise   solution $(X_t(x),t\ge 0)$ to \eqref{eq.Levy} which moreover defines a strong Markov process. Moreover, a.s. $(X_t(x),t\ge 0)\in \mathbb D(\mathbb R_+, \mathbb R^d)$.
 \end{cor}

 \begin{proof} The proof  leading  Equation  \eqref{eq.tauR} below is rather standard. We do it for sake of completeness. 
  The computations carried out in the proof of Proposition \ref{pr.C3}  shows that  for any $t\ge0$, $\int_0^t|\mathscr L^X\mathbf V(x_s)|ds$ is a.s. finite for any càdlàg process $(x_s,s\ge 0)$ (as such a process is a.s. bounded over $[0,t]$). Note also that since the number of  jumps  is at most countable, it holds a.s.  $\int_0^t \mathscr L^X\mathbf V(x_{s^-})ds=\int_0^t\mathscr L^X\mathbf V(x_s) ds$.
Let   $c_1>0$   such that $\mathscr L^X\mathbf V \le c_1\mathbf V $ over $\mathbb R^d$. Set for $R\ge 0$, $\sigma_R:=\inf\{t\ge 0, \mathbf V(X_t)\ge R\}$. Note that  $  \mathscr V_R:=\inf\{t\ge 0,  X_t \notin  \mathscr V_R\}$ where $  \mathscr V_R=  \{x\in \mathbb R^d, \mathbf V(x)<R\}$
 is an open bounded (say by $c_R>0$) subset of  $\mathbb R^d$.  Consider the unique strong solution of \eqref{eq.Levy} up to time $\sigma_R$ and set $\sigma_\infty:=\lim_{R\to +\infty} \sigma_R=\sup_{R>0}\sigma_R$ which is  its lifetime. Let us prove that $\mathbb P_x[\sigma_\infty=+\infty]=1$ for all $x\in \mathbb R^d$.
For $x\in    \mathscr V_R$, using Itô formula~\cite[Theorem 4.4.7]{applebaum2009levy} on the interval $[0, t\wedge\sigma_R(x)]$,  one has:
\begin{align*}
M_{t\wedge\sigma_R}^\mathbf V(x)  &:=\mathbf V(X_{t\wedge \sigma_R}(x))-\mathbf V(x)- \int_0^{t\wedge \sigma_R} \mathscr L^X\mathbf V(X_{s^-}(x)) ds\\
&= \int_0^{t\wedge \sigma_R} \int_{\mathbb R^d} [\mathbf V(X_{s^-}(x)+z)-\mathbf V(X_{s^-}(x))] \tilde N_\alpha(ds,dz)\\
&= \int_0^{t}  \int_{\mathbb R^d} \mathbf 1_{s< \sigma_R}[\mathbf V(X_{s^-}(x)+z)-\mathbf V(X_{s^-}(x))] \tilde N_\alpha(ds,dz),
 \end{align*}
 where $\tilde N_\alpha$ is the compensated Poisson random measure of $(L^\alpha_t,t\ge 0)$.
Since $\mathbf V(y)\le L(2+|y|^{\beta \theta})$ and $2\beta\theta<1$ (see \eqref{eq.a1}), we have for all $Y\in \mathbb R^d$,
\begin{align*}
&\int_{|z|>1}  \mathbf 1_{\mathbf V(Y)\le R }\,  |\mathbf V(Y+z)-\mathbf V(Y)|^2\,  F_\alpha (dz)\\
&\le 2 \int_{|z|>1}   \mathbf 1_{\mathbf V(Y)\le R } \mathbf V^2(Y+z)\, F_\alpha (dz) + 2 R^2 \int_{|z|>1}\,  F_\alpha (dz)\\
&\le 2L^2 \int_{|z|>1}  [4+ 2c_R^{2\beta \theta} +2|z|^{2\beta \theta}] F_\alpha (dz)+ 2R^2  \int_{|z|>1} \, F_\alpha (dz).
 \end{align*}
Moreover,  since $|\nabla \mathbf V|$ is   bounded over $\mathbb R^d$, there exists $C>0$ such that for all $Y\in \mathbb R^d$,
\begin{align*}\int_{|z|\le1} \!\!\mathbf 1_{\mathbf V(Y)\le R }|\mathbf V(Y+z)-\mathbf V(Y)|^2 F_\alpha (dz)
&\le C \int_{|z|\le1} | z|^2   F_\alpha (dz)<+\infty.
 \end{align*}
Hence, because $\mathbf 1_{s< \sigma_R}\le \mathbf 1_{\mathbf V(X_{s^-})\le R }$, one deduces that  $\mathbb E_x [\int_0^{t}  \int_{\mathbb R^d} \mathbf 1_{s< \sigma_R}|\mathbf V(X_{s^-}+z)-\mathbf V(X_{s^-})|^2 F_\alpha (dz)ds ]<+\infty$.
Consequently,  for all $R>0$, the process  $(M_{t\wedge\sigma_R}^\mathbf V(x),t\ge 0)$ is a martingale.
Thus,
\begin{align*}
\mathbb E_{x}[\mathbf V(X_{t\wedge\sigma_R})]&\le \mathbf V(x) + \mathbb E_x\Big[\int_0^{t\wedge \sigma_R} \mathscr L^X\mathbf V(X_{s^-}) ds\Big] \\
&\le \mathbf V(x) + c_1\, \mathbb E_x\Big[\int_0^{t\wedge \sigma_R}  \mathbf V(X_{s^-}) ds\Big] \\
&=\mathbf V(x) +c_1\, \mathbb E_x\Big[\int_0^{t\wedge \sigma_R}  \mathbf V(X_{s}) ds\Big]  \le\mathbf V(x) +c_1\, \mathbb E_x\Big[\int_0^{t}  \mathbf V(X_{s\wedge \sigma_R}) ds\Big].
\end{align*}
By Grönwall's inequality~\cite[Theorem~5.1 in Appendixes]{EK}, we   have   $\mathbb E_{x}[\mathbf 1_{\sigma_R\le t} \mathbf V(X_{t\wedge\sigma_R})]\le \mathbb E_{x}[\mathbf V(X_{t\wedge\sigma_R})]\le\mathbf V(x)  e^{c_1t}$. Since $\mathbf V(X_{\sigma_R}) \ge R$, it then holds \begin{equation}\label{eq.tauR}
\mathbb P_x[\sigma_R\le t]\le \frac{e^{c_1t} }{R} \,  \mathbf V(x), \ \forall t\ge 0.
\end{equation}
This proves that $\sigma_\infty(x)$ is a.s. infinite. Note also that  we have proved that  $\mathbf V\in \mathbb D_e(\mathfrak L^X)$ and hence  $\mathbf W\in \mathbb D_e(\mathfrak L^X)$    (see~\cite[Proposition 5.1]{guillinqsd}).  The strong Markov property follows from   the (pathwise) uniqueness by standard considerations.
 \end{proof}

 We also mention~\cite[Theorem 3.1]{ZhangSIMA} for existence and uniqueness of  solutions to $\alpha$-stable driven stochastic differential equations in a similar setting where the analysis   relies  there on the fact that a rotationally invariant $\alpha$-stable process  is (in law) a subordinated Brownian motion. We also refer to~\cite{kulik2009} where the exponential ergodicity of a Markov process defined as the solution to a SDE with jump noise is investigated.

In the following, we denote by $(P_t,t\ge 0)$ the semigroup of the process $(X_t,t\ge 0)$, where we recall that $P_tf(x)=\mathbb E_x[f(X_t)]$,
 for all $f\in b\mathcal B (\mathbb R^d)$.


  \begin{prop}\label{pr.SF}
  For all $t>0$,   $P_t$ is  strong Feller.
  \end{prop}

\begin{proof}
 Let $R>0$ and  let $ U_R:\mathbb R^d\to \mathbb R$ be a $\mathcal C^2$ function  such that $U_R=  U$ on the closure of $ \mathscr V_R$ and such that all its derivatives of order less than $2$ are bounded. Let $(X^R_s,s\ge 0)$ be the solution of $dX^R_t= -\nabla U_R(X^R_t)dt + dL_t^\alpha$. For all $R>0$, by~\cite[Theorem 1.1]{zhang2013derivative} and standard approximation theorem~\cite[Lemma 7.1.5]{da1996ergodicity}, $x\in \mathbb R^d\mapsto \mathbb E_x[f(X_t^R)]$ is continuous for all $t>0$ and $f\in b\mathcal B(\mathbb R^d)$.  Let $x_n\to x\in \mathbb R^d$. Consider a compact set $K$ and  $R_0>0$ such that $x_n,x\in K\subset \mathscr V_{R_0}$ for all $n\ge 1$.
 Since $(X_s,s\ge 0)$ and $(X^R_s,s\ge 0)$ coincides before their first exit time from $\mathscr V_R$, we have  for all $f\in b\mathcal B(\mathbb R^d)$, $t\ge 0$, $R\ge R_0$, and $n\ge 1$,
\begin{align*}
&\big|\mathbb E_x[f(X_t)]-\mathbb E_{x_n}[f(X_t)]\big|\\
 &\le  \big|\mathbb E\big [[f(X^R_t(x))-f(X^R_t(x_n))]\mathbf 1_{t< \sigma_R(x)\wedge  \sigma_R(x_n)}\big ]\big|+ 2\Vert f\Vert \, \mathbb P[\sigma_R(x)\wedge  \sigma_R(x_n)\le t]\\
 &\le  \big|\mathbb E\big [[f(X^R_t(x))-f(X^R_t(x_n))\big ]\big|+ 4\Vert f\Vert \, \mathbb P[\sigma_R(x)\wedge  \sigma_R(x_n)\le t].
 \end{align*}
 Note that $\sup_{n\ge 1}\mathbb P[\sigma_R(x)\wedge  \sigma_R(x_n)\le t]\le \mathbb P[\sigma_R(x) \le t]+\sup_{n\ge 1}\mathbb P[   \sigma_R(x_n)\le t]\le 2 e^{c_1t}\sup_{K}\mathbf V /R\to 0 $ as $R\to +\infty$ by \eqref{eq.tauR}. Hence $|\mathbb E_x[f(X_t)]-\mathbb E_{x_n}[f(X_t)]|\to 0$ as $n\to +\infty$, the desired result.
\end{proof}

Let us mention that the  strong Feller property of solutions to L\'evy driven stochastic differential equation has been extensively   investigated in the literature, see e.g.~\cite{picard1996existence,priola2012exponential,kusuoka2014smoothing,dong2016strong,zhang2017fundamental,xi2019jump,xi2019jump} and references therein.

 \subsubsection{On Assumption {\rm\textbf{(C2)}}}

 \begin{prop}\label{pr.TF}    The process $(X_t,t\ge 0)$ solution to \eqref{eq.Levy} satisfies {\rm\textbf{(C2)}}.
  \end{prop}

 \begin{proof}
 Let $T>0$ be fixed.
Recall that for  $r>0$,  $  \mathscr V_R=  \{x\in \mathbb R^d, \mathbf V(x)<R\}$
 is an open bounded subset of  $\mathbb R^d$  and  $\sigma_R= \inf\{t\ge 0, X_t \notin \mathscr V_R\}$.
Let $x_n\to x\in \mathbb R^d$.   One can assume that for some $R_0>0$,  $x_n,x\in  \mathscr V_R$ for all $R\ge R_0$ and $n\ge 0$. In the following we assume that $R\ge  R_0$. By Gronwall Lemma~\cite[Theorem~5.1 in Appendixes]{EK} we have for all $R\ge 0$, when $T<\sigma_R(x)\wedge \sigma_R(x_n)$,  $\sup_{s\in [0,T]}| X_t(x)-X_t(x_n)|\le |x-x_n|e^{b_R T}$,
where $b_R:= \sup_{y\in \mathscr H_R} |\text{Hess }U(y)|$. Thus,  one has for all $\epsilon>0$ and all $R>0$ fixed, as $n\to \infty$,
$$\mathbf  P_n(R)=\mathbb P\Big [\sup_{s\in [0,T]}| X_t(x)-X_t(x_n)| \ge \epsilon, T< \sigma_R(x)\wedge \sigma_R(x_n)\Big]\to 0.$$
Consequently, it holds for any $\epsilon >0$:
\begin{align*}
 \mathbb P\big [\sup_{s\in [0,T]}| X_t(x)-X_t(x_n)| \ge \epsilon\big] &\le \mathbf  P_n(R) + \mathbb P\big [T\ge  \sigma_R(x)\wedge \sigma_R(x_n)]\\
&\le \mathbf  P_n(R) + \mathbb P\big [T\ge  \sigma_R(x) ]+ \mathbb P\big [T\ge   \sigma_R(x_n)]\\
&\le \mathbf  P_n(R)+ 2R_0\frac{e^{c_1T} }{R},
\end{align*}
where we have used \eqref{eq.tauR} to get the last inequality and the fact that $\mathbf V(x_n)< R_0$ and $\mathbf V(x)< R_0$. Let us now consider $\delta>0$. Pick $R_\delta\ge R_0$ such that $2R_0 {e^{c_1T} }/{R_\delta}\le \delta/2$. For this fix $R_\delta>0$, $\mathbf  P_n(R_\delta)\to 0$ as $n\to +\infty$, and thus, there exists $N_\delta\ge 1$ such that  for all $n\ge N_\delta$,  $\mathbf  P_n(R_\delta)\le \delta/2$. Therefore, one has that  $\mathbb P  [\sup_{s\in [0,T]}| X_t(x)-X_t(y)| \ge \epsilon ] \le \delta$ for all $n\ge N_\delta$.
We have thus proved that $X_{[0,T]}(x_n)\to X_{[0,T]}(x)$ in probability as $n\to +\infty$ for the supremum norm  over  $[0,T]$. Thus  $X_{[0,T]}(x_n)\to X_{[0,T]}(x)$  in probability  as $n\to +\infty$ also for the distance of $ \mathbb D([0,T],\mathbb R^d)$, see e.g.~\cite[Section 12]{billingsley2013}.
 Therefore,  $\mathbb P_{  x_n}[X_{[0,T]}\in \cdot]$  converges weakly to $\mathbb P_{  x}[X_{[0,T]}\in \cdot]$ in
 $\mathcal P(\mathbb D([0,T],\mathbb R^d))$, which is precisely   \textbf{(C2)}. The proof is thus complete.
 \end{proof}

\subsubsection{On Assumption {\rm\textbf{(C5)}}}

 Let $ \mathscr  D$ be an open subset   of $\mathbb R^d$. Recall that    $(P_t^{ \mathscr D},t\ge 0)$ denotes the semigroup of the killed process $(X_t,t\ge 0)$, i.e.  $P_t^{ \mathscr D}f(x)=\mathbb E_x[f(X_t)\mathbf 1_{t<\sigma_{  \mathscr D}}]$,
 for all $f\in b\mathcal B ( \mathscr D)$ (see \eqref{killedsg}).
We have the following result.
 \begin{prop}\label{pr.Topo}
 Let $ \mathscr  D$ be a  subdomain of $\mathbb R^d$. 
 Let  $T>0$ and $x,z\in \mathscr D$. Then,  for all $\epsilon>0$ and $T>0$,
 \begin{equation}\label{eq.Topo}
 P_T^{ \mathscr D}(x, B(z,\epsilon))=\mathbb P_x[|X_T-z|<\epsilon, T<\tau_{\mathscr D}]>0.
 \end{equation}
 In addition, if $\mathbb R^d\setminus \overline{\mathscr D}$ is nonempty, then for all $y\in \mathscr D$, $\mathbb P_y[\sigma_{\mathscr D}<+\infty ]>0$.
 Thus, Assumption {\rm\textbf{(C5)}} is satisfied.
  \end{prop}


\begin{proof}
Fix $T>0$ and $\epsilon>0$.  Let $\mathscr  O$ be a   bounded  subdomain of $\mathscr D$ with closure included in $\mathscr D$ and  such that $x,z\in \mathscr O$.
The proof is divided into   three steps.
\medskip

\noindent
\textbf{Step A.} Preliminary analysis.
    Let $\mathbf c:\mathbb R^d\to \mathbb R^d$ be a globally Lipschitz vector field  such that $\mathbf c=-\nabla U$ on the closure of $ \mathscr  O$. Let $(X^\star_s,s\ge 0)$ be the solution of $dX^\star_t= \mathbf c(X^\star_t)dt + dL_t^\alpha$.
 Since $(X_s,s\ge 0)$ and $(X^\star_s,s\ge 0)$ coincides before their first exit time from $ \mathscr O$, Equation \eqref{eq.Topo} holds if
 \begin{equation}\label{eq.K}
 \mathbb P_x\big [\{|X^\star_T-z|<\epsilon\} \cap \{\overline{\rm Ran} \, X^\star_{[0,T]}\subset \mathscr  O\}\big ]>0.
 \end{equation}
Let us prove  \eqref{eq.K}.  In the rest of the proof, we adopt the notation of~\cite[Section 2.2]{kulyk22}.  Note here that  $r\equiv 0$, $b=\mathbf c$, $\sigma\equiv 1$, and $c(x,u)\equiv u$.   In view of~\cite[Theorem 2.1]{kulyk22} and \eqref{eq.K}, the goal is to construct   $\phi\in  \mathbf S^{\text{const}}_{0,T,x}$ such that
\begin{equation}\label{eq.Kulyk0}
\text{$\overline{\rm  Ran}\, \phi \subset \mathscr O$, $\phi_0=x$,  and $|\phi_T-z|<\epsilon/2$.}
\end{equation}
If such a curve $\phi$ exists, by~\cite[Theorem 2.1]{kulyk22}, it holds   for all  $\epsilon_0>0$,
\begin{equation}\label{eq.Kulyk}
 \mathbb P_x[ d_T(X^\star,\phi)<\epsilon_0]>0,
 \end{equation}
 where $d_T$ is the Skorokhod metric  of $ \mathbb D([0,T], \mathbb R^d)$, see~\cite[Section 12]{billingsley2013}. It is not difficult to construct    $\phi\in   \mathbf S^{\text{const}}_{0,T,x}$   satisfying \eqref{eq.Kulyk0} using the  simple procedure described in~\cite[Equation (7)]{kulyk22}.

\medskip

\noindent
\textbf{Step B.} Construction of the curve $\phi$.
  We will construct $\phi$ with $f_t\equiv 0$.  First note that for any $r'>0$ and $a\neq b\in \mathbb R^d$, $J(a,B(b,r'))=F(B(b-a,r'))\in (0,+\infty)$ if $r'<|b-a|$. Therefore $a\neq b\Rightarrow b\in \text{supp}(J(a,\cdot))$ (i.e. the jump from $a$ to $b$ is admissible). Define
 $$\tilde{\mathbf b}(y)=\mathbf c(y)- \mathbf a,  \text{ where $\mathbf a=-\int_{|w|\le 1}w_L F_\alpha(dw)$},$$
 and $w_L$ is the orthogonal projection of $z$ on the vector space  $L=  \{\ell \in \mathbf R^d, \int_{|w|\le 1}|w\cdot \ell| F_\alpha(dw)<+\infty\}$.

 Assume $x\neq z$ (the  case $x=z$  is treated similarly). Consider two disjoint open balls  $B(x,\epsilon')$ and $B(z,\epsilon')$ whose closures are included in $\mathscr O$ for any $\epsilon'\in (0,\epsilon'_0)$ for some $\epsilon'_0\in (0,\epsilon/4)$. Fix such a $\epsilon'>0$.
 \medskip

 \noindent
 \textit{Initialization}.
 Let $\phi$ be the solution of $\dot \phi_s=\tilde{\mathbf b}(\phi_s)$ with $\phi_0=x$.
 Choose $t_1\in (0,T)$ such that $\phi_t\in \overline B(x,\epsilon'/2)$ for all $t\in  [0,t_1)$ and $\phi_{t_1^-}\in \overline B(x,\epsilon'/2)$.
Then, pick $x_1\in B(z,\epsilon'/2)$    and set $\phi_{t_1}=x_1\neq x$ (this jump  is admissible according to the previous discussion).
 \medskip

 \noindent
 \textit{Second step}.
We let $\phi_t$  evolve again according   to the flow  $\dot u_s=\tilde{\mathbf b}(u_s)$   on $[t_1,t_2)$ with initial condition $x_1$. If $\tilde{\mathbf b}(x_1)= 0$, we stop the procedure because  $\phi_t=x_1$ for all $t\ge t_1$. Otherwise, there exists $t_2>t_1$  such that  $\phi_t\in \overline B(z,\epsilon'/2)$ for all $t\in  [t_1,t_2)$ and $x_1\neq \phi_{t_2^-}\in  \overline B(z,\epsilon'/2)$. If one can choose $t_2>T$, we stop the construction of $\phi$. Otherwise, we come back to $x_1$ setting
 $$\phi_{t_2}:=x_1$$
 and we then consider this point  as the initial value of the Cauchy problem $\dot u_s=\tilde{\mathbf b}(u_s)$ on the time interval $[t_2,t_2+t_2-t_1)$.
 \medskip

 \noindent
 \textit{Iteration}.
Then, one repeats this procedure  a finite number of times to construct $\phi$ over $[0,t_1)\cup [t_1,t_2)\cup [t_2, 2t_2-t_1)\cup [2t_2-t_1,  3t_2-2t_1)\cup\ldots [nt_2-(n-1)t_1,(n+1)t_2-nt_1)$. By choice of $\epsilon'>0$, the resulting $\phi$ has the desired properties, i.e. $\phi\in  \mathbf S^{\text{const}}_{0,T,x}$ satisfies \eqref{eq.Kulyk0}.
\medskip

\noindent
\textbf{Step C.}  End of the proof of \eqref{eq.K}.
 Assume that $\epsilon_0>0$ is small enough (say $\epsilon_0\in (0,\epsilon_\phi)$, $\epsilon_\phi\in (0,\epsilon)$) such that $ d_T(f,\phi)<\epsilon_0/2$ implies
 that $\overline{\rm Ran} \, f\subset \mathscr O$ (note  that  in particular $|f_T-z|\le |f_T-\phi_T| + |\phi_T-z|\le  d_T(f,\phi)+\epsilon/2<\epsilon$).  Then, using  \eqref{eq.Kulyk} with such a small  $\epsilon_0>0$ yields  $\mathbb P_x[\{|X^\star_T-z|<\epsilon\} \cap \{ \overline{\rm Ran}\, X^\star_{[0,T]}\subset  \mathscr O\}]>0$, which  is exactly \eqref{eq.K}.
 Therefore,  \eqref{eq.Topo} is satisfied.     The second statement in Proposition \ref{pr.Topo} is then   easy to obtain with the same analysis.
This concludes the proof of Proposition \ref{pr.Topo}.
  \end{proof}

\subsubsection{On Assumption {\rm\textbf{(C4)}}}

 \begin{lem}\label{pr.C5}
For all compact subset $K$ of $\mathbb R^d$ and $\delta>0$,
 $ \lim_{t\to 0^+} \sup_{x\in K}\mathbb P_x[\sigma_{B(x,\delta)}\le t]=0$.
\end{lem}


One way to prove  Lemma \ref{pr.C5} is to study    the trajectories of the process \eqref{eq.Levy}, as this done in the proof of \cite[Lemma 2.4]{guillinqsd3}.  We  give  here another proof based on the Itô formula which is  inspired from the computations leading to \eqref{eq.tauR}.

\begin{proof}
Let $K$ be a compact subset of $\mathbb R^d$.
Let $\Psi:\mathbb R^d\to [0,1]$ be a smooth function such that  $\Psi=0$ on $B(0,\rho/2)$ and $\Psi=1$ on $B^c(0,\rho)$.  Note that $\Psi$, $\nabla \Psi$ and ${\rm Hess}\, \Psi$ are uniformly bounded over $\mathbb R^d$. Set  $\Psi_x(z)= \Psi(z-x)$.
Let $K_\delta$ be the (closed)  $\delta$-neighborhood of $K$.
In addition, for any $x,z\in \mathbb R^d$,
\begin{align*}
|\mathscr L^X\Psi_x(z)|&\le |\nabla U(z)\cdot \nabla_z\Psi_x(z)|+ 2\int_{|y|>1}F_\alpha(dy)+ \sup_{w\in \mathbb R^d} |{\rm Hess}_w\, \Psi_x(w)|\, \int_{|y|\le 1}|y|^2F_\alpha(dy)\\
&\le |\nabla U(z)\cdot \nabla_z\Psi_x(z)| + C_0,
\end{align*}
for some $C_0>0$ independent of $x,z\in \mathbb R^d$.
Since for $x\in K$, $\nabla_z\Psi_x(z)=0$  for all $z\notin K_\delta$,  we deduce that $\sup_{x\in K, z\in K_\delta}|\mathscr L^X\Psi_x(z)|<+\infty$.
On the other hand,   by Itô formula, the  process   $(M_{t}^{\Psi_x}(x),t\ge 0)\text{ is a martingale}$,
where $M_{t}^{\Psi_x}(x):=\Psi_x(X_{t }(x))-\Psi_x(x)- \int_0^{t } \mathscr L^X\Psi_x(X_{s^-}(x)) ds$.  In particular $\Psi_x\in \mathbb D_e(\mathfrak L^X)$ and $\mathfrak L^X\Psi_x=\mathscr L^X\Psi_x$. Thus, since in addition $\Psi_x(x)=0$, we have using the optional stopping theorem,
\begin{align*}
\mathbb E_{x}[\Psi_x(X_{t\wedge\sigma_{B(x,\delta)}})]&\le   \mathbb E_x\Big[\int_0^{t\wedge \sigma_{B(x,\delta)}} \mathscr L^X\Psi_x(X_{s^-}(x)) ds\Big]  \\
&\le t \sup_{x\in K, z\in K_\delta}|\mathscr L^X\Psi_x(z)|.
\end{align*}
Notice that we have used above that  when $x\in K$ and $s<\sigma_{B(x,\delta)}$, $X_{s^-}(x)\in K_\delta$ and thus  for all $x\in K$ and $s<\sigma_{B(x,\delta)}$, $|\mathscr L^X\Psi_x(X_{s^-}(x))|\le \sup_{x\in K, z\in K_\delta}|\mathscr L^X\Psi_x(z)|$.
Note that $|X_{ \sigma_{B(x,\delta)}}(x)-x|\ge \delta$.
Hence,  $\Psi_x(X_{ \sigma_{B(x,\delta)}}(x))=1$ and
$$ \mathbb P_x[\sigma_{B(x,\delta)}\le t]= \mathbb E_{x}[\mathbf 1_{\sigma_{B(x,\delta)}\le t} \Psi_x(X_{ \sigma_{B(x,\delta)}})]\le t \sup_{x\in K, z\in K_\delta}|\mathscr L^X\Psi_x(z)|.$$
This ends the proof of the lemma.
\end{proof}

Using \textbf{(C1)}, Lemma \ref{pr.C5}, and the same arguments as those used to prove~\cite[Theorem~2.2]{chung2001brownian}, we deduce the following result.
 \begin{cor}\label{co.SFD}
 Let $ \mathscr D$ be any non empty open subset  of $\mathbb R^d$ and $t>0$.
Then, for any $t>0$, $P_t^{\mathscr  D}$  is strongly Feller. In particular {\rm\textbf{(C4)}} holds.
 \end{cor}

 In conclusion, 
we have thus proved that, when  $\beta>1$ and   \eqref{eq.Uh} holds, the process $(X_t,t\ge 0)$ solution to \eqref{eq.Levy} satisfies \textbf{(C1)}  $\to$ \textbf{(C5)} with the Lyapunov function defined in \eqref{eq.LW-S}. Hence, we have the following result.

\begin{thm}\label{th.LL}
Let $\beta>1$ and assume \eqref{eq.Uh}. Let $\mathscr D$ be any subdomain of $\mathbb R^d$ such that $\mathbb R^d\setminus \overline{\mathscr D}$ is nonempty.
Then, the empirical distribution of the process solution to \eqref{eq.Levy} (see Corollary \ref{co.Ex}) satisfies all the assertions of Theorem~\ref{thm-main1},  Theorem~\ref{thm-main2}  and Corollary~\ref{co1}   with the Lyapunov function defined in~\eqref{eq.LW-S}.
\end{thm}
\medskip

\noindent
\textbf{Note}. When $\mathscr D$ is bounded, one can modify  $U$ outside $\overline{\mathscr D}$ so that it satisfies  \eqref{eq.Uh} and then, as $\mathbf W$ is bounded over $\overline{\mathscr D}$ (see \eqref{eq.LW-S}),  all the assertions  of Theorem~\ref{thm-main1},  Theorem~\ref{thm-main2}  and Corollary~\ref{co1} hold on the whole space $b\mathcal B(\mathscr D)$ and for all $\nu \in \mathcal P(\mathscr D)$.  
\medskip

\noindent
\textbf{Note}. Notice that the fact that $\mathscr D$ is connected is not necessary to get the result of Proposition~\ref{pr.Topo} (this is the main difference with solutions to SDE driven by a  Brownian motion), and thus Theorem \ref{th.LL}   holds   when e.g.  $\mathscr D$ is a finite union of disjoint subdomains of $\mathbb R^d$ such that $\mathbb R^d\setminus \overline{\mathscr D}$ is nonempty.

 \medskip

\noindent
 \textbf{Acknowledgement.}\\
 {\small A. Guillin is supported by the ANR-23-CE-40003, Conviviality, and has benefited from a government grant managed by the Agence Nationale de la
Recherche under the France 2030 investment plan ANR-23-EXMA-0001. 
B.N. is  supported by  the grant  IA20Nectoux from the Projet I-SITE Clermont CAP 20-25 and   by  the ANR-19-CE40-0010, Analyse Quantitative de Processus M\'etastables (QuAMProcs).}

{\small
 \bibliography{GrandeD} 

\begin{thebibliography}{10}

\bibitem{applebaum2009levy}
D.~Applebaum.
\newblock {\em L{\'e}vy {P}rocesses and {S}tochastic {C}alculus}.
\newblock Cambridge University Press, 2009.

\bibitem{ascione2024bulk}
G.~Ascione and J.~L{\H{o}}rinczi.
\newblock Bulk behaviour of ground states for relativistic {S}chr{\"o}dinger
  operators with compactly supported potentials.
\newblock In {\em Annales Henri Poincar{\'e}}, volume~25, pages 2941--2994.
  Springer, 2024.

\bibitem{benaim2021degenerate}
M.~Bena{\"\i}m, N.~Champagnat, W.~O{\c{c}}afrain, and D.~Villemonais.
\newblock Degenerate processes killed at the boundary of a domain.
\newblock {\em Preprint arXiv:2103.08534}, 2021.

\bibitem{billingsley2013}
P.~Billingsley.
\newblock {\em Convergence of {P}robability {M}easures}.
\newblock John Wiley \& Sons, 2013.

\bibitem{bogdan-book}
K.~Bogdan, T.~Byczkowski, T.~Kulczycki, M.~Ryznar, R.~Song, and Z.~Vondracek.
\newblock {\em Potential analysis of stable processes and its extensions}.
\newblock Springer Science \& Business Media, 2009.

\bibitem{breyer1999quasi}
L.A. Breyer and G.O. Roberts.
\newblock A quasi-ergodic theorem for evanescent processes.
\newblock {\em Stochastic processes and their applications}, 84(2):177--186,
  1999.

\bibitem{carmona1990relativistic}
R.~Carmona, W.C. Masters, and B.~Simon.
\newblock Relativistic {S}chr{\"o}dinger operators: asymptotic behavior of the
  eigenfunctions.
\newblock {\em Journal of Functional Analysis}, 91(1):117--142, 1990.

\bibitem{cat1}
P.~Cattiaux, P.~Collet, A.~Lambert, S.~Mart\'inez, S.~M\'el\'eard, and
  J.~San~Mart\'in.
\newblock Quasi-stationary distributions and diffusion models in population
  dynamics.
\newblock {\em Ann. Probab.}, 37(5):1926--1969, 2009.

\bibitem{cat2}
P.~Cattiaux and S.~M\'el\'eard.
\newblock Competitive or weak cooperative stochastic {L}otka-{V}olterra systems
  conditioned on non-extinction.
\newblock {\em J. Math. Biol.}, 60(6):797--829, 2010.

\bibitem{champagnat2019probabilistic}
N.~Champagnat and B.~Henry.
\newblock A probabilistic approach to {D}irac concentration in nonlocal models
  of adaptation with several resources.
\newblock {\em The Annals of Applied Probability}, 29(4):2175--2216, 2019.

\bibitem{champagnat2024quasi}
N.~Champagnat, T.~Leli{\`e}vre, M.~Ramil, J.~Reygner, and D.~Villemonais.
\newblock Quasi-stationary distribution for kinetic {SDE}s with low regularity
  coefficients.
\newblock {\em Preprint arXiv:2410.01042}, October 2024.

\bibitem{champagnat2021lyapunov}
N.~Champagnat and D.~Villemonais.
\newblock Lyapunov criteria for uniform convergence of conditional
  distributions of absorbed {M}arkov processes.
\newblock {\em Stochastic Processes and their Applications}, 135:51--74, 2021.

\bibitem{champagnat2017general}
N.~Champagnat and D.~Villemonais.
\newblock General criteria for the study of quasi-stationarity.
\newblock {\em Electronic Journal of Probability}, 28:1--84, 2023.

\bibitem{chazottes2016sharp}
J-R. Chazottes, P.~Collet, and S.~M{\'e}l{\'e}ard.
\newblock Sharp asymptotics for the quasi-stationary distribution of
  birth-and-death processes.
\newblock {\em Probability Theory and Related Fields}, 164(1-2):285--332, 2016.

\bibitem{chen2020large}
K.~Chen, Z-Q.and~Tsuchida.
\newblock Large deviation for additive functionals of symmetric {M}arkov
  processes.
\newblock {\em Transactions of the American Mathematical Society},
  373(4):2981--3005, 2020.

\bibitem{chen2015intrinsic}
X.~Chen and J.~Wang.
\newblock Intrinsic contractivity properties of {F}eynman-{K}ac semigroups for
  symmetric jump processes with infinite range jumps.
\newblock {\em Frontiers of Mathematics in China}, 10:753--776, 2015.

\bibitem{chen2016intrinsic}
X.~Chen and J.~Wang.
\newblock Intrinsic ultracontractivity of {F}eynman-{K}ac semigroups for
  symmetric jump processes.
\newblock {\em Journal of Functional Analysis}, 270(11):4152--4195, 2016.

\bibitem{chen2000intrinsic}
Z-Q. Chen and R.~Song.
\newblock Intrinsic ultracontractivity, conditional lifetimes and conditional
  gauge for symmetric stable processes on rough domains.
\newblock {\em Illinois Journal of Mathematics}, 44(1):138--160, 2000.

\bibitem{chung2001brownian}
K.L. Chung and Z.~Zhao.
\newblock {\em From Brownian {M}otion to Schr{\"o}dinger’s {E}quation},
  volume 312.
\newblock Springer Science \& Business Media, 2001.

\bibitem{collet2012quasi}
P.~Collet, S.~Mart{\'\i}nez, and J.~San~Mart{\'\i}n.
\newblock {\em Quasi-{S}tationary {D}istributions: {M}arkov {C}hains,
  {D}iffusions and {D}ynamical {S}ystems}.
\newblock Springer Science \& Business Media, 2012.

\bibitem{collet2024branching}
P.~Collet, S.~M{\'e}l{\'e}ard, and J.~San~Martin.
\newblock Branching diffusion processes and spectral properties of
  {F}eynman-{K}ac semigroup.
\newblock {\em Preprint arXiv:2404.09568}, 2024.

\bibitem{da1996ergodicity}
G.~Da~Prato and J.~Zabczyk.
\newblock {\em Ergodicity for {I}nfinite {D}imensional {S}ystems}, volume 229.
\newblock Cambridge university press, 1996.

\bibitem{daubechies1983one}
I.~Daubechies and E.H. Lieb.
\newblock One-electron relativistic molecules with {C}oulomb interaction.
\newblock {\em Communications in Mathematical Physics}, 90(4):497--510, 1983.

\bibitem{del2004feynman}
P.~Del~Moral.
\newblock {\em Feynman-{K}ac {F}ormulae}.
\newblock Probability and Its Applications. Springer, 2004.

\bibitem{del2023stability}
P.~Del~Moral, E.~Horton, and A.~Jasra.
\newblock On the stability of positive semigroups.
\newblock {\em The Annals of Applied Probability}, 33(6A):4424--4490, 2023.

\bibitem{del2003particle}
P.~Del~Moral and L.~Miclo.
\newblock Particle approximations of {L}yapunov exponents connected to
  {S}chr{\"o}dinger operators and {F}eynman--{K}ac semigroups.
\newblock {\em ESAIM: Probability and Statistics}, 7:171--208, 2003.

\bibitem{DS89}
J-D. Deuschel and D.W. Stroock.
\newblock {\em Large {D}eviations}.
\newblock Academic Press, Boston, 1989.

\bibitem{di-gesu-lelievre-le-peutrec-nectoux-17}
G.~Di~Ges{\`u}, T.~Leli{\`e}vre, D.~Le~Peutrec, and B.~Nectoux.
\newblock Jump {M}arkov models and transition state theory: the
  quasi-stationary distribution approach.
\newblock {\em Faraday Discussions}, 195:469--495, 2017.

\bibitem{DLLN}
G.~Di~Ges\`u, T.~Leli\`evre, D.~Le~Peutrec, and B.~Nectoux.
\newblock Sharp asymptotics of the first exit point density.
\newblock {\em Annals of PDE}, 5(2), 2019.

\bibitem{dong2016strong}
Z.~Dong, X.~Peng, Y.~Song, and X.~Zhang.
\newblock Strong {F}eller properties for degenerate {SDE}s with jumps.
\newblock {\em Annales de l'IHP Probabilit{\'e}s et statistiques},
  52(2):888–--897, 2016.

\bibitem{EK}
S.~N. Ethier and T.G. Kurtz.
\newblock {\em Markov {P}rocesses: {C}haracterization and {C}onvergence}.
\newblock John Wiley \& Sons, 1986.

\bibitem{ferre2}
G.~Ferr{\'e}, M.~Rousset, and G.~Stoltz.
\newblock More on the long time stability of {F}eynman--{K}ac semigroups.
\newblock {\em Stochastics and Partial Differential Equations: Analysis and
  Computations}, 9(3):630--673, 2021.

\bibitem{guillinFK}
A.~Guillin, D.~Lu, B.~Nectoux, and L.~Wu.
\newblock Long time behavior of killed {F}eynman-{K}ac semigroups with singular
  {S}chr\"{o}dinger potentials.
\newblock {\em Preprint Hal-04790621}, 2024.

\bibitem{guillinqsd3}
A.~Guillin, D.~Lu, B.~Nectoux, and L.~Wu.
\newblock Generalized {L}angevin and {N}os{\'e}-{H}oover processes absorbed at
  the boundary of a metastable domain.
\newblock {\em Preprint arXiv:2403.17471}, March 2024.

\bibitem{guillinqsd}
A.~Guillin, B.~Nectoux, and L.~Wu.
\newblock Quasi-stationary distribution for strongly {F}eller {M}arkov
  processes by {L}yapunov functions and applications to hypoelliptic
  {H}amiltonian systems.
\newblock {\em Journal of the European Mathematical Society}, 26(8):3047--3090,
  2022.

\bibitem{guillinqsd2}
A.~Guillin, B.~Nectoux, and L.~Wu.
\newblock Quasi-stationary distribution for {H}amiltonian dynamics with
  singular potentials.
\newblock {\em Probability Theory and Related Fields}, 185(3-4):921--959, 2023.

\bibitem{guneysu2011feynman}
B.~G{\"u}neysu.
\newblock {\em On the {F}eynman-Kac formula for {S}chr{\"o}dinger semigroups on
  vector bundles}.
\newblock PhD thesis, Universit{\"a}ts-und Landesbibliothek Bonn, 2011.

\bibitem{he2019some}
G.~He, G.~Yang, and Y.~Zhu.
\newblock Some conditional limiting theorems for symmetric {M}arkov processes
  with tightness property.
\newblock {\em Electronic {C}ommunications in {P}robability}, 24(60):1--11,
  2019.

\bibitem{kaleta2010intrinsic}
K.~Kaleta and T.~Kulczycki.
\newblock Intrinsic ultracontractivity for {S}chr{\"o}dinger operators based on
  fractional {L}aplacians.
\newblock {\em Potential Analysis}, 33:313--339, 2010.

\bibitem{Kato}
T.~Kato.
\newblock {\em Perturbation {T}heory for {L}inear {O}perators}.
\newblock Classics in Mathematics. Springer-Verlag, Berlin, 1995.
\newblock Reprint of the 1980 edition.

\bibitem{kim2016large}
D.~Kim, K.~Kuwae, and Y.~Tawara.
\newblock Large deviation principles for generalized {F}eynman-{K}ac
  functionals and its applications.
\newblock {\em Tohoku Mathematical Journal, Second Series}, 68(2):161--197,
  2016.

\bibitem{kim2024quasi}
D.~Kim, T.~Tagawa, and A.~Velleret.
\newblock Quasi-ergodic theorems for {F}eynman-{K}ac semigroups and large
  deviation for additive functionals.
\newblock {\em Preprint arXiv:2401.17997}, 2024.

\bibitem{kulczycki2006intrinsic}
T.~Kulczycki and B.~Siudeja.
\newblock Intrinsic ultracontractivity of the {F}eynman-{K}ac semigroup for
  relativistic stable processes.
\newblock {\em Transactions of the American Mathematical Society},
  358(11):5025--5057, 2006.

\bibitem{kulik2009}
A.M. Kulik.
\newblock Exponential ergodicity of the solutions to {SDE}’s with a jump
  noise.
\newblock {\em Stochastic Processes and their Applications}, 119(2):602--632,
  2009.

\bibitem{kulyk22}
O.~Kulyk.
\newblock Support theorem for {L}{\'e}vy-driven stochastic differential
  equations.
\newblock {\em Journal of Theoretical Probability}, pages 1--23, 2022.

\bibitem{kusuoka2014smoothing}
S;~Kusuoka and C;~Marinelli.
\newblock On smoothing properties of transition semigroups associated to a
  class of {SDE}s with jumps.
\newblock In {\em Annales de l'IHP Probabilit{\'e}s et statistiques},
  volume~50, pages 1347--1370, 2014.

\bibitem{lelievre2022eyring}
T.~Leli{\`e}vre, D.~Le Le~Peutrec, and B.~Nectoux.
\newblock Eyring-kramers exit rates for the overdamped {L}angevin dynamics: the
  case with saddle points on the boundary.
\newblock {\em Preprint arXiv:2207.09284}, 2022.

\bibitem{ramilarxiv2}
T.~Leli{\`e}vre, M.~Ramil, and J.~Reygner.
\newblock Quasi-stationary distribution for the {L}angevin process in
  cylindrical domains, part {I}: existence, uniqueness and long-time
  convergence.
\newblock {\em Stochastic Processes and their Applications}, 144:173--201,
  2022.

\bibitem{lelievre2016partial}
T.~Leli{\`e}vre and G.~Stoltz.
\newblock Partial differential equations and stochastic methods in molecular
  dynamics.
\newblock {\em Acta Numerica}, 25:681--880, 2016.

\bibitem{lladser2006domain}
M.~Lladser and J.~San~Mart{\'{\i}}n.
\newblock Domain of attraction of the quasi-stationary distributions for the
  {O}rnstein-{U}hlenbeck process.
\newblock {\em Journal of {A}pplied {P}robability}, 37(2):511--520, 2016.

\bibitem{MR92}
Z-M. Ma and M.~R{\"o}ckner.
\newblock {\em Introduction to the theory of (non-symmetric) {D}irichlet
  forms}.
\newblock Springer Science \& Business Media, 2012.

\bibitem{meleard2012quasi}
S.~M{\'e}l{\'e}ard and D.~Villemonais.
\newblock Quasi-stationary distributions and population processes.
\newblock {\em Probability Surveys}, 9:340--410, 2012.

\bibitem{MT1993}
S.~P. Meyn and R.~L. Tweedie.
\newblock {\em Markov {C}hains and {S}tochastic {S}tability}.
\newblock Communications and Control Engineering Series. Springer-Verlag
  London, 1993.

\bibitem{nagasawa2012stochastic}
M.~Nagasawa.
\newblock {\em Stochastic {P}rocesses in {Q}uantum {P}hysics}, volume~94.
\newblock Birkh{\"a}user, 2012.

\bibitem{picard1996existence}
Jean Picard.
\newblock On the existence of smooth densities for jump processes.
\newblock {\em Probability Theory and Related Fields}, 105:481--511, 1996.

\bibitem{priola2012exponential}
E.~Priola, A.~Shirikyan, L.~Xu, and J.~Zabczyk.
\newblock Exponential ergodicity and regularity for equations with {L}{\'e}vy
  noise.
\newblock {\em Stochastic Processes and their Applications}, 122(1):106--133,
  2012.

\bibitem{rousset2006control}
M.~Rousset.
\newblock On the control of an interacting particle estimation of
  {S}chr{\"o}dinger ground states.
\newblock {\em SIAM journal on {M}athematical {A}nalysis}, 38(3):824--844,
  2006.

\bibitem{simon2015quantum}
B.~Simon.
\newblock {\em Quantum {M}echanics for {H}amiltonians {D}efined as {Q}uadratic
  {F}orms}, volume~72.
\newblock Princeton University Press, 2015.

\bibitem{Wu97}
L.~Wu.
\newblock An introduction to large deviations (in chinese).
\newblock pages 225--336, 1997. In: Several Topics in Stochastic Analysis
  (authors: J.A. Yan, S.Peng, S. Fang and L. Wu), Academic Press of China,
  Beijing.

\bibitem{Wu2001}
L.~Wu.
\newblock Large and moderate deviations and exponential convergence for
  stochastic damping {H}amiltonian systems.
\newblock {\em Stochastic Processes and their Applications}, 91(2):205--238,
  2001.

\bibitem{Wu2004}
L.~Wu.
\newblock Essential spectral radius for {M}arkov semigroups. {I}. {D}iscrete
  time case.
\newblock {\em Probability Theory and Related Fields}, 128(2):255--321, 2004.

\bibitem{xi2019jump}
F.~Xi and C.~Zhu.
\newblock Jump type stochastic differential equations with non-{L}ipschitz
  coefficients: non-confluence, {F}eller and strong {F}eller properties, and
  exponential ergodicity.
\newblock {\em Journal of Differential Equations}, 266(8):4668--4711, 2019.

\bibitem{Yosida}
K.~Yosida.
\newblock Functional analysis, 1980.
\newblock {\em Spring-Verlag, New York/Berlin}, 1971.

\bibitem{zhang2014quasi}
J.~Zhang, S.~Li, and R.~Song.
\newblock Quasi-stationarity and quasi-ergodicity of general {M}arkov
  processes.
\newblock {\em Science China Mathematics}, 57:2013--2024, 2014.

\bibitem{zhang2013derivative}
X.~Zhang.
\newblock Derivative formulas and gradient estimates for {SDE}s driven by
  $\alpha$-stable processes.
\newblock {\em Stochastic Processes and their Applications}, 123(4):1213--1228,
  2013.

\bibitem{ZhangSIMA}
X.~Zhang.
\newblock Fundamental {S}olution of {K}inetic {F}okker--{P}lanck {O}perator
  with {A}nisotropic {N}onlocal {D}issipativity.
\newblock {\em SIAM Journal on Mathematical Analysis}, 46(3):2254--2280, 2014.

\bibitem{zhang2017fundamental}
X.~Zhang.
\newblock Fundamental solutions of nonlocal {H}ormander's operators {II}.
\newblock {\em The Annals of Probability}, 45(3):1799--1841, 2017.

\end{thebibliography}

\bibliographystyle{plain}

}

%
%
%
%
%
%
%
%
%
%
%
%
%
%
%
%
%

\end{document}